\documentclass[11pt]{article}

\usepackage{graphicx} \usepackage{amsmath,amssymb,bm,amsthm,mathrsfs,amscd, mathtools, url}
\usepackage{here}

\usepackage{fancyhdr}
\usepackage{tikz}
\usepackage{color}
\usepackage[inline]{enumitem}
\usepackage{xparse}

\usepackage{hyperref}

\theoremstyle{plain}
\newtheorem{thm}{Theorem}[section]
\newtheorem{lem}[thm]{Lemma}
\newtheorem{cor}[thm]{Corollary}
\newtheorem{prop}[thm]{Proposition}

\theoremstyle{definition}
\newtheorem{dfn}[thm]{Definition}
\newtheorem{eg}[thm]{Example}

\theoremstyle{remark}
\newtheorem{rem}[thm]{Remark}

\usepackage{authblk}

\allowdisplaybreaks

\newcommand{\dist}{\mathop{d_{\mathrm{R}}}}

\DeclareMathOperator{\Max}{Max}
\DeclareMathOperator{\Min}{Min}

\DeclareMathOperator{\rank}{rank}
\DeclareMathOperator{\supp}{supp}
\DeclareMathOperator{\Supp}{Supp}
\DeclareMathOperator{\Tr}{Tr}
\DeclareMathOperator{\wt}{wt}

\numberwithin{equation}{section}

\newcommand*{\Mat}[3]{\operatorname{Mat}_{#1 \times #2}(#3)}
\newcommand*{\qbinom}[3]{\genfrac{[}{]}{0pt}{}{#1}{#2}_{#3}}
\newcommand{\set}[1]{\left\{#1\right\}}
\renewcommand{\P}{P}
\renewcommand{\mid}{:}
\newcommand{\uniform}[3]{\mathcal{U}_{#1, #2}^{(#3)}}

\newcommand*{\supps}[2]{\mathscr{S}_{#1}(#2)}
\newcommand*{\minsupps}[2]{s_{#1}(#2)}

\newcommand{\abs}[1]{\left\lvert #1 \right\rvert}

\NewDocumentCommand{\qmpolyhigherweights}{ o m g }{V_{#2}\IfValueT{#3}{^{(#3)}}}

\NewDocumentCommand{\qmpolyhigherweightsg}{ o m g }{\mathcal{V}_{#2}\IfValueT{#3}{^{[#3]}}}

\NewDocumentCommand{\W}{ o m g }{W_{#2}\IfValueT{#1}{^{(#1)}}\IfValueT{#3}{\!\left(#3\right)}}

\NewDocumentCommand{\A}{ o m g }{A_{#2}\IfValueT{#1}{^{(#1)}}\IfValueT{#3}{\!\left(#3\right)}}
\NewDocumentCommand{\B}{ o m g }{B_{#2}\IfValueT{#1}{^{(#1)}}\IfValueT{#3}{\!\left(#3\right)}}
\NewDocumentCommand{\N}{ o m g }{N_{#2}\IfValueT{#1}{^{(#1)}}\IfValueT{#3}{\!\left(#3\right)}}

\newcommand{\rev}[1]{#1}
\newcommand{\del}[1]{} \title{Higher Rank-Support Weights and \del{$q$}\rev{$(q,m)$}-Polymatroids}
\author[1]{Koji Imamura\thanks{k-imamura@imi.kyushu-u.ac.jp}}
\author[2]{Shinya Kawabuchi\thanks{k.shinya.math.eng@gmail.com}}
\author[3]{Keisuke Shiromoto\thanks{keisuke@kumamoto-u.ac.jp}}

\affil[1]{Institute of Mathematics for Industry (IMI), Kyushu University
744 Motooka, Nishi-ku, Fukuoka 819-0395, Japan}
\affil[2]{Graduate School of Science and Technology, Kumamoto University
2-39-1, Kurokami, Kumamoto 860-8555, Japan}
\affil[3]{Faculty of Advanced Science and Technology, Kumamoto University
2-39-1, Kurokami, Kumamoto 860-8555, Japan} 
\begin{document}

\maketitle

\begin{abstract}

The aim of this paper is to develop a $(q,m)$-polymatroidal approach to higher supports and higher rank-weight enumerators of Delsarte rank-metric codes. In this framework, we establish rank-metric analogs of several fundamental results known for matroids and linear codes. We describe minimal higher supports in terms of cocircuits of suitable truncations of the associated $(q,m)$-polymatroid, and obtain a polymatroidal interpretation of generalized minimum rank weights and their Singleton-type bound. We also prove that the associated $(q,m)$-polymatroid and the higher support distributions determine each other. On the enumerative side, we establish a Greene-type identity relating higher rank-weight enumerators to the rank generating function, and conversely show that the full sequence of higher rank-weight enumerators determines this function. As the main application, we prove MacWilliams-type identities for higher rank-weight enumerators by first deriving a duality formula for higher weight polynomials of $(q,m)$-polymatroids and then applying it to a rank-metric code and its dual.
\end{abstract} 
\section{Introduction} \label{sect:Introduction}

The relationship between linear codes and matroids is one of the classical links between coding theory and combinatorics. Given a linear code, one obtains an associated matroid from a generator matrix, and this matroid reflects the dependence structure of the coordinate positions. Through this correspondence, several fundamental parameters of the code may be interpreted in matroidal terms. Most notably, the minimum distance of the code is equal to the minimum size of the cocircuits of the associated matroid. This observation already illustrates that the distance properties of a code are closely tied to the combinatorial structure of the corresponding matroid.

A further manifestation of this relationship is Greene's identity, due to Greene~\cite{Greene}. His result states that the weight enumerator of a linear code is determined by the Tutte polynomial of its associated matroid. Accordingly, the weight distribution of the code can be understood through a distinguished invariant from matroid theory.

The matroidal viewpoint extends naturally from codewords to subcodes of fixed dimension.
For a linear code, one may consider the supports of all subcodes of a prescribed dimension,
and these higher-dimensional supports have likewise been studied from a matroid-theoretic perspective.
In particular, Britz showed that the minimal supports arising in this way are described by cocircuits of suitable truncations of the matroid associated with the code \cite{HigherSupportMatroid}.
Among the corresponding enumerative invariants, the higher weight enumerators are of particular interest.
Whereas the ordinary weight enumerator is obtained only as a specialization of the Tutte polynomial
and therefore does not, in general, determine it,
the full family of higher weight enumerators contains exactly the same information as the Tutte polynomial.
More precisely, the Tutte polynomial determines all higher weight enumerators,
and conversely the collection of higher weight enumerators determines the Tutte polynomial \cite{HigherSupportWeight}.

Rank-metric codes are linear codes equipped with the rank distance, with applications ranging
from random network coding to space-time coding.
In the Delsarte model\rev{~\cite{Delsarte1978}}, one studies $\mathbb{F}_{q}$-linear subspaces $\mathcal{C} \leq \Mat{n}{m}{\mathbb{F}_{q}}$,
in contrast to the Gabidulin model\rev{~\cite{Gabidulin1985}} based on $\mathbb{F}_{q^{m}}$-linear subspaces of $\mathbb{F}_{q^{m}}^{n}$.
In this work we focus on the Delsarte setting and develop higher polynomial invariants
through $q$-analogs of polymatroids.

A key combinatorial approach was introduced by Shiromoto \cite{Shiromoto2019},
who associates to a Delsarte rank-metric code $\mathcal{C} \leq \Mat{n}{m}{\mathbb{F}_{q}}$
a $(q,m)$-polymatroid $P_{\mathcal{C}} = (E,\rho)$ on $E=\mathbb{F}_{q}^{n}$
and defines a four-variable rank generating function $R_{P_{\mathcal{C}}}$.
This framework yields a Greene-type identity expressing the rank-weight enumerator of
a Delsarte rank-metric code in terms of this polynomial invariant \cite{Shiromoto2019}.
In a different, though closely related, direction, Gluesing-Luerssen and Jany showed that,
for representable $q$-matroids, the Whitney function and the full family of higher weight enumerators
of any representing \rev{$\mathbb{F}_{q^{m}}$-linear} rank-metric code determine each other via monomial substitutions;
moreover, the $q$-matroid itself and the family of higher support distributions
likewise determine each other \cite{GLJ25MonomialSubstitution}.
\del{By contrast, for Delsarte rank-metric codes and their associated $(q,m)$-polymatroids,
an analogous equivalence between higher support or weight data and
the corresponding polynomial invariant does not appear to have been made explicit
in the existing literature.
The purpose of the present paper is to provide such a correspondence in the framework of
Delsarte rank-metric codes and $(q,m)$-polymatroids,
from the perspective of higher support distributions and higher rank-weight enumerators.}

\rev{Singleton-type bounds for generalized rank weights have been
established in several related frameworks.
Ravagnani~\cite[Theorem~30.5]{RAVAGNANI2016} extended generalized
rank weights to $\mathbb F_q$-linear matrix codes using optimal
anticodes and proved a Singleton-type bound for the resulting
weights.
By \cite[Theorem~37]{demipoly2020}, these weights coincide with the column-support weights used here when $m>n$.
Mart{\'i}nez-Pe{\~n}as and
Matsumoto~\cite{MartinezPenasMatsumoto2018}
introduced relative generalized matrix weights for nested pairs
of matrix codes and established a Singleton-type bound that
recovers the bound for generalized rank weights.
Britz, Mammoliti, and
Shiromoto~\cite{BritzMammolitiShiromoto2020}
derived the bound from monotonicity properties of generalized
weights of $(q,m)$-demimatroids.
Alfarano, Borello, and Neri~\cite{ABN2026}
provided a geometric proof of the bound using a characterization
of generalized rank weights in terms of the slice spaces
of generator tensors.}

\rev{Like \cite{GLJ25MonomialSubstitution}, Molina~\cite{Molina2026} considers $\mathbb{F}_{q^{m}}$-linear rank-metric codes and counts $\mathbb{F}_{q^{m}}$-linear subcodes.
In this setting, he established MacWilliams-type identities in terms of generalized binomial moments.
The same paper also expresses generalized rank-weight enumerators using subcode counts in shortened codes and determines the generalized rank-weight distributions of MRD codes.
Here we consider arbitrary $\mathbb{F}_{q}$-linear matrix codes and $\mathbb{F}_{q}$-subcodes, and derive related identities using $(q,m)$-polymatroids.}

\rev{The MacWilliams-type identities of Molina and of Byrne, Cotardo, and Ravagnani~\cite{GeneralizedRankWeight} express duality in terms of generalized binomial moments and weight-distribution coefficients.}

\rev{For Hamming-metric codes, Kaplan~\cite{Kaplan2014} studied MacWilliams identities for tuple weight enumerators and described their relation to higher weight enumerators.
We use a $q$-analogue of this relation for $(q,m)$-polymatroids.}

\rev{Our contributions are threefold.
First, we characterize the minimal higher rank supports of subcodes of a fixed dimension as cocircuits of suitable truncations of the associated $(q,m)$-polymatroid.
We also describe how the rank function and the higher support distributions determine each other.
Second, we use Shiromoto's rank generating function to establish a Greene-type identity for higher rank-weight enumerators and give an explicit inverse formula recovering the rank generating function from their full family.
Third, we establish an explicit Kl{\o}ve-type MacWilliams formula for the higher weight polynomials of arbitrary $(q,m)$-polymatroids, without assuming representability (Theorem~\ref{thm:Klove-type-MacWilliams-identity}). Specialization to a code and its dual gives a $q$-transform formula for the higher rank-weight enumerators of arbitrary $\mathbb{F}_{q}$-linear matrix codes (Corollary~\ref{cor:Klove-type-MacWilliams-identity}). The related duality identities in \cite{GeneralizedRankWeight,Molina2026} are expressed through generalized binomial moments and weight-distribution coefficients. Here we give an explicit transformation of the higher weight polynomials and derive it within the polymatroidal framework. To the best of our knowledge, this particular $q$-transform formulation has not previously appeared.}

The organization of the paper is as follows.
\del{In Section~\ref{sect:Preliminaries}, we review the necessary background on the relationship
between Delsarte rank-metric codes and $(q, m)$-polymatroids.
In Section~\ref{sect:HigherRankSupportCodes}, we observe that the cocircuits of the $(q, m)$-polymatroid
associated with a given rank-metric code correspond to the minimal higher supports of the code.
We also show that the associated $(q,m)$-polymatroid determines the higher support distributions,
and conversely that the latter determine the $(q,m)$-polymatroid.
In Section~\ref{sect:Greene}, we establish a Greene-type identity for the higher rank-weight enumerators.
In Section~\ref{sect:Application}, we use the Greene-type identity established in Section 4 to derive a MacWilliams-type identity for the higher rank-weight enumerators, and to determine the higher rank-weight enumerators of MRD codes.}
\rev{Section~\ref{sect:Preliminaries} reviews rank-metric codes, $(q,m)$-polymatroids, and their rank generating functions.
Section~\ref{sect:HigherRankSupportCodes} studies minimal higher supports, generalized minimum rank weights, and higher support distributions.
Section~\ref{sect:Greene} develops the Greene-type identity and the inverse formula recovering the rank generating function from higher rank-weight enumerators.
Section~\ref{sect:Application} derives a MacWilliams-type identity using the $q$-transform and recovers the known higher rank-weight distribution of MRD codes.} 
\section{Preliminaries} \label{sect:Preliminaries}

\rev{We first define supports, weights, and duals of Delsarte rank-metric codes.
We then introduce $(q,m)$-polymatroids and their rank generating functions as defined in~\cite{Shiromoto2019}. Finally, we recall the known Greene-type identity expressing the ordinary rank-weight enumerator in terms of the rank generating function.}

Throughout this paper, let $R[x_1,\dots, x_l]$ be the polynomial ring in the indeterminates
  $x_1, \dots, x_l$ with coefficients in a ring $R$. 
  Let $n$, $m$ be fixed positive integers and
denote by $\Mat{n}{m}{\mathbb{F}_{q}}$ the $\mathbb{F}_{q}$-vector space
consisting of the $n \times m$ matrices with entries in $\mathbb{F}_{q}$.
Let $E \coloneqq \mathbb{F}_{q}^{n}$ be the $n$-dimensional column vector space.
We denote by $\mathcal{L}(E)$ the collection of all subspaces of $E$.
If $X, Y \in \mathcal{L}(E)$ and $X$ is a subspace of $Y$, we write $X \leq Y$. 
For any subset $S$ of an $\mathbb{F}_q$-vector space, we denote by $\langle S \rangle$ the $\mathbb{F}_q$-linear span of $S$. We denote the zero space by $0$.
For a collection $\mathcal{A}$ of subsets of a set, we denote the set of minimal (resp. maximal)
members in $\mathcal{A}$ under the inclusion relation by $\Min(\mathcal{A})$ (resp. $\Max(\mathcal{A})$).

\subsection{Delsarte Rank-Metric Code}

We first endow the matrix space $\Mat{n}{m}{\mathbb{F}_{q}}$ with the distance $\dist$
defined as $\dist(X,Y)\coloneqq \rank(X - Y)$ for all $X, Y \in \Mat{n}{m}{\mathbb{F}_{q}}$.
It is easy to see that $\dist(X, 0) = \rank X$.
A (Delsarte) \emph{rank-metric code} $\mathcal{C}$ is a subspace of $\Mat{n}{m}{\mathbb{F}_{q}}$.
Each matrix in $\mathcal{C}$ is called a \emph{codeword}, and
every $\mathbb{F}_{q}$-subspace $\mathcal{D} \leq \mathcal{C}$ is called
a \emph{subcode} of $\mathcal{C}$.
Throughout this paper, $\mathcal{C}$ denotes a rank-metric code,
and $k$ denotes the dimension of $\mathcal{C}$.

\begin{dfn} \label{dfn:support-weight}
  For every codeword $M \in \mathcal{C}$, we denote by $\supp(M)$ the column space of $M$,
  and call it the \emph{rank-support} of $M$.
  For every subcode $\mathcal{D} \leq \mathcal{C}$, 
  \begin{equation*} 
    \Supp(\mathcal{D}) \coloneqq \sum_{M \in \mathcal{D}}\supp(M)
  \end{equation*}
  is called the \emph{higher rank-support} of $\mathcal{D}$.
  The \emph{rank-weight} of $M \in \mathcal{C}$ and
  the \emph{higher rank-weight} of $\mathcal{D} \leq \mathcal{C}$ are defined as
  \begin{equation*}
    \wt(M) = \dim \supp(M)
    \quad \text{and} \quad
    \wt(\mathcal{D}) = \dim \Supp(\mathcal{D}),
  \end{equation*}
  respectively.
\end{dfn}

For any subspace $J \in \mathcal{L}(E)$, define
\begin{align*}
  \mathcal{C}(J) &\coloneqq \{ M \in \mathcal{C} \mid \supp(M) \subseteq J \}, \\
  J^{\perp} &\coloneqq \{ \bm{y} \in E \mid \bm{x} \cdot \bm{y} = 0 \text{ for all } \bm{x} \in J \},
\end{align*}
where $\bm{x} \cdot \bm{y}$ is the standard inner product over $\mathbb{F}_q$.
We remark that $\mathcal{C}(J)$ is an $\mathbb{F}_{q}$-subspace of $\Mat{n}{m}{\mathbb{F}_{q}}$ (cf. \cite[Lemma~1]{Shiromoto2019}).
For $\mathscr{X} \subseteq \mathcal{L}(E)$, we write $\mathscr{X}^{\perp}\coloneqq\set{X^{\perp}\mid X \in \mathscr{X}}$.
For all $M, N \in \Mat{n}{m}{\mathbb{F}_{q}}$, we define the \emph{inner product} of $M$ and $N$ as
\begin{equation*}
  M \cdot N \coloneqq \Tr(MN^{\top}), 
\end{equation*}
where $\Tr(\cdot)$ denotes the trace map on square matrices, and $N^{\top}$ is the transpose of $N$.
Then the \emph{dual} of $\mathcal{C}$ is defined as
\begin{equation*}
  \mathcal{C}^{\perp} \coloneqq \{ N \in \Mat{n}{m}{\mathbb{F}_{q}} \mid M \cdot N = 0 \text{ for all } M \in \mathcal{C} \}.
\end{equation*}
 
\begin{dfn} \label{dfn:rank_distribution}
  Let $\mathcal{C}$ be a rank-metric code in $\Mat{n}{m}{\mathbb{F}_{q}}$.
  The \emph{rank-weight distribution} of $\mathcal{C}$ is a function
  $A_{\mathcal{C}} \colon \{ 0, 1, \dots, n \} \to \mathbb{Z}$ defined as
  \begin{equation*}
    \A{\mathcal{C}}{w} \coloneqq \abs{\{M \in \mathcal{C} \mid \wt(M) = w \}}
    \qquad \text{for all } w \in \{ 0, 1, \dots, n \}.
  \end{equation*}
\end{dfn}

\subsection{$(q, m)$-Polymatroid and the Greene-Type Identity}

\begin{dfn} \label{dfn:qm_polymatroid}
  Fix a nonnegative integer $\del{r}\rev{\lambda}$.
  A \emph{$(q, \del{r}\rev{\lambda})$-polymatroid} is an ordered pair $\P = (E, \rho)$ consisting of
  a finite-dimensional vector space $E$ over $\mathbb{F}_{q}$ and
  a function $\rho \colon \mathcal{L}(E) \to \mathbb{Z}_{\geq 0}$ having the following properties:
  \begin{enumerate}[label=\textup{(R\arabic*)}, itemsep=0ex]
    \item \label{r1} If $A \leq E$, then $0 \leq \rho(A) \leq \del{r}\rev{\lambda}\dim A$.
    \item \label{r2} If $A, B \leq E$ and $A \leq B$, then $\rho(A) \leq \rho(B)$.
    \item \label{r3} If $A , B\leq E$, then $\rho(A + B) + \rho(A \cap B) \leq \rho(A) + \rho(B)$.
  \end{enumerate}
  Furthermore, we call $\rho(E)$ the \emph{rank} of $\P$.
\end{dfn}

The following is one of the most fundamental $(q, m)$-polymatroids.

\begin{dfn}[\cite{q-uniform, rankmetric_q-poly}]
  \label{dfn:uniform_qm_polymatroid}
  Let $\ell$ be a nonnegative integer with $\ell \leq n$.
  If $\rho \colon \mathcal{L}(E) \to \mathbb{Z}_{\geq 0}$ is defined as
  \begin{equation*}
    \rho(J)\coloneqq m \min \set{\dim J, \ell}\qquad(J \in \mathcal{L}(E)),
  \end{equation*}
  we call the pair $(E, \rho)$ a \emph{uniform} $(q, m)$-polymatroid of rank $m\ell$ and denote it by $\uniform{\ell}{n}{m}$.
\end{dfn}

\begin{dfn}[\cite{Shiromoto2019}] \label{dfn:rank_generating_function}
  Let $\P$ be a $(q, m)$-polymatroid. We define the \emph{rank generating function} $R_{\P}$ of $\P$ as follows:
  \begin{equation*}
    R_{\P}(X_{1}, X_{2}, X_{3}, X_{4}) \coloneqq \sum_{A \in \mathcal{L}(E)} X_{1}^{\rho(E)-\rho(A)}X_{2}^{m\dim A - \rho(A)}g_{\dim A}(X_{3}, X_{4}),
  \end{equation*}
  where, with the convention that the empty product is $1$,
  \begin{equation*}
    g_{\ell}(X, Y) \coloneqq \prod_{i=0}^{\ell- 1}(X-q^{i}Y) \qquad (\ell \in \mathbb{Z}_{\geq 0}).
  \end{equation*}
\end{dfn}

\begin{rem}
  One may also consider the two-variable Whitney rank generating function
  \begin{equation*}
    \widetilde{R}_{\P}(X,Y)
      \coloneqq
      \sum_{J\in \mathcal{L}(E)} X^{\rho(E)-\rho(J)}Y^{m\dim J-\rho(J)}.
  \end{equation*}
  This polynomial contains exactly the same information as
  the four-variable rank generating function $R_{\P}(X_{1}, X_{2}, X_{3}, X_{4})$.
  Indeed, since $g_{\ell}(1,0) = 1$ for every $\ell \geq 0$, one has
  \begin{equation*}
    \widetilde{R}_{\P}(X,Y) = R_{\P}(X, Y, 1, 0).
  \end{equation*}
  Conversely, if
  \begin{equation*}
    \widetilde{R}_{\P}(X,Y)
    =\sum_{a=0}^{\rho(E)} \sum_{b=0}^{nm - \rho(E)} c_{a,b} X^{a} Y^{b},
  \end{equation*}
  then
  \begin{equation*}
    R_{\P}(X_{1}, X_{2}, X_{3}, X_{4})
     =
     \sum_{a=0}^{\rho(E)} \sum_{b=0}^{nm-\rho(E)} c_{a,b}\,
     X_{1}^{a} X_{2}^{b} \,
     g_{\frac{\rho(E)-a+b}{m}}(X_3,X_4),
  \end{equation*}
  where $(\rho(E) - a + b)/m \in \mathbb{Z}_{\geq 0}$ whenever $c_{a,b} \neq 0$.
  Thus passing from $R_{\P}$ to $\widetilde{R}_{\P}$ does not lose any information.
  
  The distinction between the two is therefore not one of information content,
  but of presentation: the four-variable polynomial keeps the $(\dim J)$-grading
  explicit, whereas the two-variable polynomial suppresses it. For the Greene-type
  and MacWilliams-type identities considered here, this grading is the natural
  one. Indeed, the Greene-type identity below is obtained from $R_{\P}$ by a direct
  specialization of the variables, while a formulation in terms of
  $\widetilde{R}_{\P}$ requires first recovering the $(\dim J)$-grading and then
  reintroducing the factor $g_{\dim J}$. For this reason, although
  $\widetilde{R}_{\P}$ is an equivalent invariant, $R_{\P}$ is better
  suited for the identities studied in this paper.
\end{rem}

Given a pair of nonnegative integers $a$ and $b$ with $a \geq b$,
the \emph{$q$-binomial} or \emph{Gaussian coefficient} counts the number of $b$-dimensional subspaces
of an $a$-dimensional subspace over $\mathbb{F}_{q}$ and is given by:
\begin{equation*}
  \qbinom{a}{b}{q} \coloneqq \prod_{i=0}^{b - 1} \frac{q^{a} - q^{i}}{q^{b} - q^{i}}.
\end{equation*}
We also use the following extension for integer upper arguments:
for $a \in \mathbb{Z}$ and $b \in \mathbb{Z}_{\geq 0}$, set
$\qbinom{a}{b}{q}
  \coloneqq 
  \prod_{i=0}^{b-1}\frac{q^{a-i}-1}{q^{b-i}-1}$,
and set it equal to $0$ if $b < 0$. We write $\qbinom{E}{j}{q}$ to denote the set of all $j$-subspaces of $E$
(the $j$-Grassmannian of $E$).

\begin{prop} \label{prop:rank_generating_function_uniform}
  The rank generating function of the uniform $(q, m)$-polymatroid $\uniform{\ell}{n}{m}$ of rank $m\ell$ is
  \begin{equation*}
    R_{\uniform{\ell}{n}{m}}(X_1,X_2,X_3,X_4)
    =
\sum_{j=0}^{\ell}
\qbinom{n}{j}{q}
X_1^{m(\ell-j)}
g_j(X_3,X_4)
+
\sum_{j=\ell+1}^{n}
\qbinom{n}{j}{q}
X_2^{m(j-\ell)}
g_j(X_3,X_4).
  \end{equation*}
\end{prop}
\begin{proof}
  \begin{align*}
R_{\uniform{\ell}{n}{m}}(X_1,X_2,X_3,X_4)
&=
\sum_{A\in\mathcal L(E)}
X_1^{m\ell -\rho(A)}X_2^{m\dim A-\rho(A)}g_{\dim A}(X_3,X_4)\\
&=
\sum_{j=0}^{n}
\qbinom{n}{j}{q}
X_1^{m\ell-m\min\{j,\ell\}}
X_2^{mj-m\min\{j,\ell\}}
g_j(X_3,X_4)\\
&=
\sum_{j=0}^{\ell}
\qbinom{n}{j}{q}
X_1^{m(\ell-j)}
g_j(X_3,X_4)
+
\sum_{j=\ell+1}^{n}
\qbinom{n}{j}{q}
X_2^{m(j-\ell)}
g_j(X_3,X_4). \qedhere
\end{align*}
\end{proof}

\begin{prop}[{\cite{Shiromoto2019}}]
  \label{prop:qm_polymatroid_from_rank_metric_code}
  Let $\mathcal{C}\leq \Mat{n}{m}{\mathbb{F}_{q}}$ be a Delsarte rank-metric code, and let $\rho_{\mathcal{C}} \colon \mathcal{L}(E) \to \mathbb{Z}_{\geq 0}$ be the function defined as
  \begin{equation*}
    \rho_{\mathcal{C}}(J) \coloneqq \rev{m\dim J - \dim \mathcal{C}^{\perp}(J) = }\dim \mathcal{C} - \dim \mathcal{C}(J^{\perp})
    \qquad\text{for all } J \in \mathcal{L}(E).
  \end{equation*}
  Then, $\P_{\mathcal{C}} \coloneqq (E, \rho_{\mathcal{C}})$ is a $(q, m)$-polymatroid.
\end{prop}

\rev{
The dimension identity in the proof of \cite[Proposition~11]{Shiromoto2019} gives
\begin{equation*}
\dim\mathcal{C}^\perp(J)+\dim\mathcal{C}
=m\dim J+\dim\mathcal{C}(J^\perp).
\end{equation*}
Rearranging yields the second equality.
}

\rev{
\begin{rem}
The expression $m\dim J-\dim\mathcal C^\perp(J)$
is a rank-metric analogue of the classical matroid rank formula
$r_C(A)=|A|-\dim C^\perp(A)$ for a linear code $C$,
where $C^\perp(A)$ denotes the subcode supported in $A$;
see \cite[Lemma~2.1(ii)]{Barg2002}.
\end{rem}
}

\rev{
\begin{dfn}[{\cite{GeneralizedRankWeight}}]\label{dfn:MRD}
  Assume $m\geq n$. Let
  $d=d(\mathcal{C})\coloneqq\min\{\rank M\mid 0\neq M\in\mathcal{C}\}$.
  If $k=m(n-d+1)$, $\mathcal{C}$ is called a \emph{maximum rank-distance (MRD) code}.
\end{dfn}
}

\begin{eg} \label{eg:rank_generating_function}
  Set
  \begin{equation*}
    M_{1}=
    \begin{pmatrix}
      1 & 0 & 0 \\
      0 & 1 & 0
    \end{pmatrix},\quad
    M_2=
    \begin{pmatrix}
      0 & 1 & 0 \\
      0 & 0 & 1
    \end{pmatrix},\quad
    M_3=
    \begin{pmatrix}
      0 & 0 & 1 \\
      1 & 1 & 0
    \end{pmatrix} \in \Mat{2}{3}{\mathbb{F}_{2}}.
  \end{equation*}
  Let $\mathcal{C} = \langle M_1,M_2,M_3\rangle \leq \Mat{2}{3}{\mathbb{F}_{2}}$ be a rank-metric code.
  Since $M_{1}, M_{2}, M_{3}$ are linearly independent, $k = \dim \mathcal{C} = 3$.
  A direct computation shows that $\supp(M) = \mathbb{F}_{2}^{2} = E$ for all nonzero codewords $M \in \mathcal{C}$.
  \rev{Its minimum rank distance is $d=2$, so $k=3=m(n-d+1)$ and $\mathcal{C}$ is an MRD code.}
  Thus, the rank of each subspace in $\mathcal{L}(E)$ is calculated as follows:
  \begin{align*}
    &\rho_{\mathcal{C}}(0) = \dim \mathcal{C} - \dim \mathcal{C}(E) = 0, \\
    &\rho_{\mathcal{C}}(\langle (1, 0)^{\top} \rangle)
    = \rho_{\mathcal{C}}(\langle (0, 1)^{\top} \rangle)
    = \rho_{\mathcal{C}}(\langle (1, 1)^{\top} \rangle)
    = \rho_{\mathcal{C}}(E) = \dim \mathcal{C} - 0 = 3.
  \end{align*}
  Consequently, the $(2, 3)$-polymatroid obtained from $\mathcal{C}$ as in Proposition~\ref{prop:qm_polymatroid_from_rank_metric_code} is
  the uniform $(2, 3)$-polymatroid $\uniform{1}{2}{3}$ of rank $3$.
  By Proposition~\ref{prop:rank_generating_function_uniform}, the rank generating function
  of $\P_{\mathcal{C}} = \uniform{1}{2}{3}$ is
  \begin{align*}
    R_{\P_{\mathcal{C}}}(X_{1}, X_{2}, X_{3}, X_{4})
    &= \qbinom{2}{0}{2} X_{1}^{3(1 - 0)} g_{0}(X_{3}, X_{4}) + \qbinom{2}{1}{2} X_{1}^{3(1-1)}g_{1}(X_{3}, X_{4}) + \qbinom{2}{2}{2} X_{2}^{3(2-1)} g_{2}(X_{3}, X_{4}) \\
    &= X_{1}^{3} + 3(X_{3} - X_{4}) + X_{2}^{3} (X_{3} - X_{4})(X_{3} - 2X_{4}).
  \end{align*}
\end{eg}

\begin{thm}[{\cite{Shiromoto2019}}]\label{thm:Greene-rank}
  \rev{Let $\mathcal{C}\leq\Mat{n}{m}{\mathbb{F}_{q}}$ be a rank-metric code.}
  Let $W_{\mathcal{C}}(x, y) \in \mathbb{Z}\lbrack x, y\rbrack$
  be the \emph{rank-weight enumerator} of $\mathcal{C}$ defined as
  \begin{equation*}
    W_{\mathcal{C}}(x, y)\coloneqq \sum_{w = 0}^{n} \A{\mathcal{C}}{w} x^{n-w}y^{w},
  \end{equation*}
  where $\A{\mathcal{C}}$ is the rank-weight distribution of $\mathcal{C}$. Then, 
  \begin{equation*}
    W_{\mathcal{C}}(x, y)= y^{n-\dim\mathcal{C}/m}R_{\P_{\mathcal{C}}}\left((qy^{1/m}), \frac{1}{y^{1/m}}, x, y\right).
  \end{equation*}
\end{thm}

\rev{
\begin{rem}
The fractional powers of $y$ cancel term by term after including the prefactor, since
\begin{equation*}
n-\frac{k}{m}+\frac{k-\rho_{\mathcal C}(U)}{m}
-\frac{m\dim U-\rho_{\mathcal C}(U)}{m}=n-\dim U\geq0.
\end{equation*}
Hence the specialized expression in Theorem~\ref{thm:Greene-rank} belongs to $\mathbb Z[x,y]$.
\end{rem}
}

\begin{eg}
  Let $\mathcal{C} = \langle M_{1}, M_{2}, M_{3} \rangle \leq \Mat{2}{3}{\mathbb{F}_{2}}$ be
  the rank-metric code in Example~\ref{eg:rank_generating_function}.
  Since all nonzero codewords in $\mathcal{C}$ have rank $2$, the weight enumerator of $\mathcal{C}$ is
  \begin{equation*}
    W_{\mathcal{C}}(x, y) = x^{2} + 7y^{2},
  \end{equation*}
  and so $\A{\mathcal{C}}{0} = 1$, $\A{\mathcal{C}}{1} = 0$, $\A{\mathcal{C}}{2} = 7$.
  Then, by Example~\ref{eg:rank_generating_function},
  \begin{align*}
    y^{n - \dim \mathcal{C}/m} R_{\P_{\mathcal{C}}}\left( (qy^{1/m}), \frac{1}{y^{1/m}}, x, y \right)
    &= y^{2 - 3/3} \left\{ (2y^{1/3})^{3} + 3(x - y) + \frac{1}{y^{3/3}}(x - y)(x - 2y) \right\} \\
    &= y \left( 8y + 3x - 3y + \frac{x^{2} -3xy + 2y^{2}}{y} \right) \\
    &= x^{2} + 7y^{2} = W_{\mathcal{C}}(x, y).
  \end{align*}
\end{eg} 
\section{Higher Rank Supports and $(q, m)$-Polymatroids} \label{sect:HigherRankSupportCodes}

\rev{In this section, we describe generalized minimum rank-weights and higher support distributions in terms of the associated $(q,m)$-polymatroid and derive the known Singleton-type bound for generalized minimum rank-weights.
We also show that the rank function determines all higher support distributions and can be recovered from the support distribution of one-dimensional subcodes.
Greene-type identities for the corresponding higher rank-weight enumerators are discussed in Section~\ref{sect:Greene}.}

\subsection{Cocircuits and Minimal Supports}
\label{subsect:CocircuitsAndMinimalSupports}

\del{First, we show that the family of minimal higher rank supports arising from subcodes of $C$ coincides with
the family of cocircuits of a truncation of the associated $(q,m)$-polymatroid.}
\rev{We show that the family of minimal higher rank supports coincides with the family of cocircuits of a truncation of the associated $(q,m)$-polymatroid.}
Throughout this subsection, we let $\P=(E,\rho)$ be a $(q,m)$-polymatroid
and $\mathcal{C}$ be a rank-metric code in $\Mat{n}{m}{\mathbb{F}_{q}}$.

\begin{dfn} \label{dfn:independent-space}
  Let $I \leq E$.
  If $\rho(I) =  m \dim I$, we call $I$ an \emph{independent} space.
  If $D \leq E$ is not independent, we call $D$ a \emph{dependent} space.
\end{dfn}

\begin{dfn} \label{dfn:circuit}
  Let $D \leq E$ be a dependent space of $\P$. 
  We call $D$ a \emph{circuit} if every proper subspace $D^{\prime} \lneq D$ is independent.
  We denote by $\mathscr{C}(\P)$ the collection of circuits of $\P$.
\end{dfn}

\begin{dfn} \label{dfn:hyperplane}
  Let $H \leq E$. If $\rho(H) < \rho(E)$ and $\rho(H')=\rho(E)$ for all $H \lneq H^{\prime} \leq E$, we call $H$ a \emph{hyperplane} of $\P$, and denote by $\mathscr{H}(\P)$
  the collection of hyperplanes of $\P$.
\end{dfn}

\begin{dfn} \label{dfn:dual}
  We call $\P^{\ast} = (E, \rho^{\ast})$ the \del{\emph{dual}}\rev{\emph{$m$-dual}} of $\P$,
  where $\rho^{\ast} \colon \mathcal{L}(E) \to \mathbb{Z}_{\geq 0}$
  is defined as $\rho^{\ast}(J) \coloneqq \rho(J^{\perp}) + m\dim{J}-\rho(E)$
  for all $J \in \mathcal{L}(E)$.
\end{dfn}
\rev{
\begin{rem}
For fixed $m$, the $m$-dual coincides with the dual defined by Shiromoto~\cite{Shiromoto2019}.
The term $m$-dual emphasizes the dependence on the parameter, which is relevant in Lemma~\ref{lem:duality-for-Pj}.
When the parameter is understood, we simply call it the dual.
\end{rem}
}
It is known that the dual of $\P$ is also a $(q,m)$-polymatroid \cite{Shiromoto2019}.
\rev{A \emph{cocircuit} of $\P$ is a circuit of $\P^{\ast}$, and we denote the collection of cocircuits by $\mathscr{C}^{\ast}(\P)\coloneqq\mathscr{C}(\P^{\ast})$.}
\rev{
For $\P_{\mathcal C}$, the rank formula in Proposition~\ref{prop:qm_polymatroid_from_rank_metric_code} yields
\begin{equation*}
\rho_{\mathcal C}^{\ast}(J)=m\dim J-\dim\mathcal C(J).
\end{equation*}
Thus $J$ is dependent in $\P_{\mathcal C}^{\ast}$ if and only if $\mathcal C(J)\ne0$.
Hence the circuits of $\P_{\mathcal C}^{\ast}$ are precisely the inclusion-minimal column supports of nonzero codewords of $\mathcal C$.
}

\begin{eg} \label{eg:cryptomorphism}
  Let $\P = \uniform{1}{2}{3}$.
  All subspaces of dimension at most $1$ are independent spaces; the ambient space $E$ is dependent.
  Thus, the collection of circuits is $\mathscr{C}(\P) = \{ E \}$, and
  the collection of hyperplanes is $\mathscr{H}(\P) = \{ 0 \}$.
  The dual $\P^{\ast}$ is the pair $(E, \rho^{\ast})$, where $\rho^{\ast} = \rho$.
\end{eg}

\del{\begin{lem}\label{lem:add-one-dimensional-space}
  For any $J \leq E$ and one-dimensional subspace $e \leq E$ with $e \cap J=0$,
  it follows that
  \begin{equation*}
    0 \leq \rho(J+e)-\rho(J) \leq m.
  \end{equation*}
\end{lem}
\begin{proof}
  By definition of a $(q, m)$-polymatroid, we obtain
  \begin{equation*}
    0 \leq\rho(J + e) - \rho(J) =  \rho(J + e) + \rho(J \cap e) - \rho(J)  \leq \rho(e) \leq m.\qedhere
  \end{equation*}
\end{proof}}

\begin{lem}[{\cite[Lemma~14]{WeightedSubspaceDesign}}]\label{lem:inherit-independency}
  Let $X \leq Y \leq E$.
  If $Y$ is an independent space, then $X$ is independent.
\end{lem}

\begin{prop}\label{prop:circuit-hyperplane}
  Let $\P$ be a $(q, m)$-polymatroid. Then, the following statement holds:
  \begin{equation*}
    \del{\mathscr{C}^{\ast}(\P) \coloneqq \mathscr{C}(\P^{\ast})=\mathscr{H}(\P)^{\perp}.}\rev{\mathscr{C}^{\ast}(\P)=\mathscr{H}(\P)^{\perp}.}
  \end{equation*}
\end{prop}

\begin{proof}
  We take $H\in \mathscr{H}(\P)$ arbitrarily. \del{By Lemma~\ref{lem:add-one-dimensional-space}}\rev{By Definition~\ref{dfn:hyperplane} and the integrality of $\rho$},
  we obtain $\rho(H)\leq \rho(E)-1$. Then,
  \begin{equation*}
    \rho^{\ast}(H^{\perp})
    = \rho(H) + m\dim(H^{\perp}) - \rho(E)
    \leq \rho(E)-1 + m\dim(H^{\perp}) -\rho(E)
    < m\dim(H^{\perp}).
  \end{equation*}
  Hence, $H^{\perp}$ is a dependent space of $\P^{\ast}$.
  \del{For any $J^{\perp}\lneq H^{\perp}$, it satisfies}\rev{All $J^{\perp}\lneq H^{\perp}$ satisfy}
  \begin{equation*}
    \rho^{\ast}(J^{\perp}) = \rho(J) + m\dim J^{\perp} - \rho(E)
    =m\dim J^{\perp}.
  \end{equation*}
  Therefore, $H^{\perp}$ is a circuit of $\P^{\ast}$.

  Conversely, we take $D \in \mathscr{C}(\P^{\ast})$ arbitrarily and have $\rho^{\ast}(D) < m \dim D$.
  Hence, 
  \begin{equation*}
  \rho(D^\perp)+m\dim D-\rho(E) < m\dim D
  \end{equation*}
  holds, and we obtain $\rho(D^{\perp}) < \rho(E)$.
  For any $H$ with $D^{\perp} \lneq H \leq E$, it holds that $H^{\perp} \lneq D$, so $H^{\perp}$ is independent in $\P^{\ast}$ by $D \in \mathscr{C}(\P^{\ast})$. Therefore, $\rho^{\ast}(H^{\perp})=m\dim H^{\perp}$. By the definition, we have $\rho(H)+ m\dim H^{\perp} - \rho(E)=m\dim H^{\perp}$ and  $\rho(H)=\rho(E)$. Hence, \(D^\perp\in \del{\mathcal H(P)}\rev{\mathscr{H}(\P)}\). Therefore
\(D=(D^\perp)^\perp\in \del{\mathcal H(P)}\rev{\mathscr{H}(\P)}^\perp\).
\end{proof}

\begin{dfn}
  For all $r \in \mathbb{Z}_{\geq 0}$, we define $\supps{\mathcal{C}}{r}$ and
  the family $\minsupps{\mathcal{C}}{r}$ of its minimal members as follows:
  \begin{equation*}
    \supps{\mathcal{C}}{r} \coloneqq \set{ \Supp(\mathcal{D}) \mid \mathcal{D} \in \qbinom{\mathcal{C}}{r}{q} },
    \quad\text{and}\quad
    \minsupps{\mathcal{C}}{r} \coloneqq \Min \supps{\mathcal{C}}{r}.
  \end{equation*}
\end{dfn}

\begin{lem}\label{lem:DCJ}
Let $\mathcal{D}\leq \mathcal{C}$ and $J \leq E$. Then, $\mathcal{D} \leq \mathcal{C}(J)$ if and only if $\Supp(\mathcal{D}) \leq J$.
\end{lem}
\begin{proof}
  \del{If $\Supp(\mathcal{D}) \leq J$, it is clear that $\supp(M) \leq \Supp(\mathcal{D}) \leq J$ for all $M \in \mathcal{D}$.
  Hence, $\mathcal{D} \leq \mathcal{C}(J)$.
  Conversely, if $\mathcal{D} \leq \mathcal{C}(J)$,
  we have $\supp(M) \leq J$ for any $M \in \mathcal{D}$.
  Hence, it follows that $\Supp(\mathcal{D})=\sum_{M \in \mathcal{D}}\supp(M) \leq J$.}
  \rev{This follows directly from the definitions.}
\end{proof}

\begin{lem}\label{lem:mimnimal-support}
For all $r \in \mathbb{Z}_{\geq 0}$,
\begin{equation*}
\minsupps{\mathcal{C}}{r} = \Min\set{J \leq E \mid \dim \mathcal{C}(J) \geq r}.
\end{equation*}
\end{lem}
\begin{proof}
If $J \in \minsupps{\mathcal{C}}{r}$, 
then $J=\Supp(\mathcal{D})$ for some $\mathcal{D} \in \qbinom{\mathcal{C}}{r}{q}$. Then, we have $\mathcal{D} \leq \mathcal{C}(\Supp \mathcal{D})=\mathcal{C}(J)$, and so $\dim \mathcal{C}(J)\geq r$. 
Suppose that there exists $J'\lneq J$ such that $\dim C(J')\ge r$.
Choose $D'\in \left[\begin{smallmatrix} C(J')\\ r\end{smallmatrix}\right]_q$.
Then, by Lemma~\ref{lem:DCJ}, 
$\Supp(D')\le J'\lneq J$, while
$\Supp(D')\in \supps{C}
{r}$, contradicting the minimality of
$J$ in $\supps{C}{r}$.

Conversely, 
assume $J \leq E$ is minimal under the condition $\dim\mathcal{C}(J) \geq r$. 
We take $\mathcal{D} \in \qbinom{\mathcal{C}(J)}{r}{q}$ arbitrarily; then $\Supp(\mathcal{D}) \leq J$ by Lemma~\ref{lem:DCJ}. 
Considering $\dim \mathcal{C}(\Supp(\mathcal{D})) \geq \dim \mathcal{D} \geq r$ and the minimality of $J$, we have $\Supp(\mathcal{D})=J$. 
Hence, we conclude that $J \in \minsupps{\mathcal{C}}{r}$.
\end{proof}

\begin{dfn} \label{dfn:truncation}
For an integer $t$ with $0 \leq t \leq \rho(E)$, we define $\rho^{(t)}\colon \mathcal{L}(E) \to \mathbb{Z}_{\geq 0}$ as
\begin{equation*}
\rho^{(t)}(J)\coloneqq\min\set{\rho(J), \, \rho(E)-t}\qquad(J \in \mathcal{L}(E)).
\end{equation*}
We call the pair $P^{(t)} \coloneqq (E, \rho^{(t)})$ the \emph{$t$-th truncation} of $\P$.
\end{dfn}

\begin{prop}
For an integer $t$ with $0 \leq t \leq \rho(E)$,
$\P^{(t)}$ is indeed a $(q, m)$-polymatroid.
\end{prop}
\begin{proof}
It is clear that $\rho^{(t)}$ satisfies \ref{r1} and \ref{r2}. 
Put $c \coloneqq \rho(E)-t$. To show \ref{r3}, let $A,B \leq E$. 
\del{If $\rho(A)\leq c$ and $\rho(B)\leq c$, we have $\rho^{(t)}(A)=\rho(A)$ and $\rho^{(t)}(B)=\rho(B)$. It follows that
\begin{equation*}
\rho^{(t)}(A+B)+\rho^{(t)}(A\cap B)
\leq
\rho(A+B)+\rho(A\cap B)
\leq
\rho(A)+\rho(B)
=
\rho^{(t)}(A)+\rho^{(t)}(B).
\end{equation*}
If $\rho(A)\geq c$ and $\rho(B)\geq c$, then $\rho^{(t)}(A)=\rho^{(t)}(B)=c$. Hence,  
\begin{equation*}
\rho^{(t)}(A+B)+\rho^{(t)}(A\cap B)\leq 2c
=
\rho^{(t)}(A)+\rho^{(t)}(B).
\end{equation*}}
\rev{The case $\rho(A)\leq c,\ \rho(B)\leq c$ is trivial.}

\del{Finally, assume that exactly one of $\rho(A),\rho(B)$ is at least $c$. By symmetry, we may assume that $\rho(A)\geq c$ and $\rho(B)<c$. Then $\rho^{(t)}(A)=c$ and $\rho^{(t)}(B)=\rho(B)$. Moreover,
\begin{equation*}
\rho^{(t)}(A+B)\leq c
\end{equation*}
and, since $A\cap B \leq B$,
\begin{equation*}
\rho^{(t)}(A\cap B)\leq \rho(A\cap B)\leq \rho(B)=\rho^{(t)}(B).
\end{equation*}
Hence
\begin{equation*}
\rho^{(t)}(A+B)+\rho^{(t)}(A\cap B)
\leq
c+\rho(B)
=
\rho^{(t)}(A)+\rho^{(t)}(B).
\end{equation*}}
\rev{Now assume that at least one of $\rho(A),\rho(B)$ is at least $c$. By symmetry, we may assume that $\rho(A)\geq c$, so $\rho^{(t)}(A)=c$. By the definition of $\rho^{(t)}$ and~\ref{r2}, we obtain
\begin{equation*}
\rho^{(t)}(A+B)+\rho^{(t)}(A\cap B)
\leq c+\rho^{(t)}(B)
=\rho^{(t)}(A)+\rho^{(t)}(B).
\end{equation*}}
Therefore $\rho^{(t)}$ satisfies {\rm (R3)}, and thus $\P^{(t)}$ is a $(q,m)$-polymatroid.
\end{proof}
The following theorem is an analog of \cite[Theorem~3]{HigherSupportMatroid} for matroids and linear codes.

\begin{thm}\label{thm:cocircuits-minimalsupports}
For $1 \leq r \leq k$,
\begin{equation*}
  \minsupps{\mathcal{C}}{r}=\mathscr{C}^{\ast}(\P_{\mathcal{C}}^{(r-1)}).
\end{equation*}
\end{thm}

\begin{proof}
By Lemma~\ref{lem:mimnimal-support} and Proposition~\ref{prop:circuit-hyperplane}, the following equations hold.
\begin{align*}
\Min \left( \set{J \leq E\mid \dim \mathcal{C}(J) \geq r} \right)
&=\Min \left( \set{J^{\perp} \leq E\mid \dim \mathcal{C}(J^{\perp}) \geq r} \right) \\
&=\Min \left( \set{J^{\perp} \leq E\mid \dim \mathcal{C} -\dim \mathcal{C}(J^{\perp}) \leq \dim \mathcal{C} -r} \right) \\
&=\Min \left( \set{J^{\perp} \leq E\mid \rho_{\mathcal{C}}(J) \leq \dim \mathcal{C} -r} \right) \\
&=\left( \Max \left( \Bigl\{ J \leq E\mid \rho_{\mathcal{C}}(J) \leq \dim \mathcal{C} -r \Bigr\} \right) \right)^{\perp}\\
&=\mathscr{H}(\P_{\mathcal{C}}^{(r-1)})^{\perp} = \mathscr{C}^{\ast}(\P_{\mathcal{C}}^{(r-1)}).\qedhere
\end{align*}
\end{proof}

\rev{We next use this correspondence to describe generalized minimum rank-weights in terms of girth.}

\del{We define the \emph{girth} of a $(q, m)$-polymatroid
as the minimum dimension of the circuits:}
\rev{We define the \emph{girth} of the $(q, m)$-polymatroid $\P$ as the minimum dimension of its circuits:}
\begin{equation*}
  g(P) \coloneqq \begin{cases}
    \infty & \text{if } \mathscr{C}(\P) = \emptyset, \\
    \min \{ \dim D : D \in \mathscr{C}(\P) \} & \text{otherwise}.
  \end{cases}
\end{equation*}

\begin{prop}\label{prop:circuitsize}
Unless $\mathscr{C}(P) = \emptyset$,
\begin{equation*}
  g(\P)
  \leq
  \left\lfloor \frac{\rho(E)}{m}\right \rfloor + 1.
\end{equation*}
\end{prop}
\begin{proof}
Since $\rho(E) < m\dim E = mn$ unless $\mathscr{C}(\P) = \emptyset$\rev{ by Lemma~\ref{lem:inherit-independency}},
we can take $D \in \qbinom{E}{\left\lfloor \rho(E)/m\right\rfloor+1}{q}$. Hence, 
\begin{equation*}
m\dim D = m \left(\left\lfloor \frac{\rho(E)}{m}\right\rfloor + 1\right) > \rho(E) \geq \rho(D).
\end{equation*}
Therefore, there exists a circuit $D'$ contained in $D$. Thus we obtain 
\begin{equation*}
g(P)\leq \dim D' \leq \dim D = \left\lfloor \frac{\rho(E)}{m}\right\rfloor + 1. \qedhere
\end{equation*}
\end{proof}

We call
\begin{equation*}
  d_{\mathcal{C}}^{(r)}
  \coloneqq
  \begin{cases}
    \infty & \text{if } \supps{\mathcal{C}}{r} = \emptyset, \\
    \min\set{\dim D \mid D \in \supps{\mathcal{C}}{r}} & \text{otherwise},
  \end{cases}
\end{equation*}
the \emph{$r$-th generalized minimum rank-weight} of $\mathcal{C}$. 
\rev{For $1\leq r\leq k$, these weights coincide with the generalized column rank weights of Alfarano, Borello, and Neri~\cite[Definition~4.5]{ABN2026}.
Their characterization via shortened codes~\cite[Theorem~4.9]{ABN2026} minimizes $\dim J$ subject to $\dim\mathcal C(J)\geq r$.
Theorem~\ref{thm:cocircuits-minimalsupports} identifies the minimal such subspaces under inclusion as cocircuits of truncations of $\P_{\mathcal C}$.
The generalized row rank weights and the corresponding results for row supports are obtained by applying our definitions and arguments to the transpose code $\mathcal{C}^{\top}$, with associated $(q,n)$-polymatroid $\P_{\mathcal{C}^{\top}}$ on $\mathbb{F}_q^m$.}

\begin{prop}\label{prop:girth-distance}
For all integers $1 \leq r \leq \dim \mathcal{C}$,
\begin{equation*}
  g\left( \left(P_{\mathcal{C}}^{(r-1)} \right)^{\ast}\right)
  =d_{\mathcal{C}}^{(r)}.
\end{equation*}
\end{prop}
\begin{proof}
It is clear that $\min\set{\dim D \mid D \in \supps{\mathcal{C}}{r}}=\min\set{\dim D \mid D \in \minsupps{\mathcal{C}}{r}}$. 
Hence, we obtain the assertion from 
Theorem~\ref{thm:cocircuits-minimalsupports} and the definition of $g\left( \left(P_{\mathcal{C}}^{(r-1)} \right)^{\ast}\right)$ and $d_{\mathcal{C}}^{(r)}$.
\end{proof}

\del{The Singleton-type bound for the generalized minimum rank-weights is known from \cite[Theorem~30.5]{RAVAGNANI2016} and \cite[Theorem~37]{demipoly2020}.}
\rev{Using girth interpretation, we recover the following known Singleton-type bound for generalized rank weights.}
\begin{cor}[{\cite[Theorem~30.5]{RAVAGNANI2016}, \cite[Theorem~37]{demipoly2020}\rev{, \cite[Theorem~10]{BritzMammolitiShiromoto2020}, \cite[Theorem~4.13]{ABN2026}}}]\label{cor:Singleton-generalized-rank}
For all integers $1 \leq r \leq \dim \mathcal{C}$, the following bound holds:
\begin{equation*}
d_{\mathcal{C}}^{(r)} \leq n -\left\lfloor\frac{\dim{\mathcal{C}} - r}{m} \right\rfloor.
\end{equation*}
\end{cor}

\begin{proof}
  By Propositions~\ref{prop:girth-distance} and \ref{prop:circuitsize},
  \begin{equation*}
    d_{\mathcal{C}}^{(r)}
    = g\left( \left( P_{\mathcal{C}}^{(r-1)} \right)^{\ast} \right)
    \leq
    \left\lfloor \frac{nm - \bigl( \dim \mathcal{C} - (r - 1) \bigr)}{m} \right\rfloor + 1
=  n -\left\lfloor\frac{\dim{\mathcal{C}} - r}{m} \right\rfloor. \qedhere
\end{equation*}
\end{proof}

\subsection{Higher Support Distributions and $(q, m)$-Polymatroids}

In this section, we show that, for a given $k$-dimensional rank-metric code $\mathcal{C} \leq \Mat{n}{m}{\mathbb{F}_{q}}$,
the associated $(q, m)$-polymatroid and the family of all higher support distributions
determine each other.
\del{As an application, we give another proof of \cite[Theorem 3.8]{GeneralizedRankWeight}
in terms of $(q, m)$-polymatroids.}

\begin{dfn}
  The \emph{$r$-th higher support distribution} of $\mathcal{C}$ is the function $\A[r]{\mathcal{C}} \colon \mathcal{L}(E) \to \mathbb{Z}_{\geq 0}$ defined as
  \begin{equation*}
    \A[r]{\mathcal{C}}{W} \coloneqq \abs{\set{\mathcal{D} \in \qbinom{\mathcal{C}}{r}{q} \mid \Supp(\mathcal{D}) = W }}
  \end{equation*}
  for all $W \in \mathcal{L}(E)$. We also define the function $\B[r]{\mathcal{C}} \colon \mathcal{L}(E) \to \mathbb{Z}_{\geq 0}$ as follows:
  \begin{equation*}
    \B[r]{\mathcal{C}}{W} \coloneqq \abs{\set{\mathcal{D} \in \qbinom{\mathcal{C}}{r}{q}\mid\Supp(\mathcal{D}) \leq W }}
  \end{equation*}
  for all $W \in \mathcal{L}(E)$.
\end{dfn}

\begin{lem} \label{lem:B-gaussian}
  For every subspace $J \leq E$,
  \begin{equation*}
    \B[r]{\mathcal{C}}(J) = \qbinom{\dim\mathcal{C}(J)}{r}{q}.
  \end{equation*}
\end{lem}
\begin{proof}
  By Lemma~\ref{lem:DCJ}, we have
  \begin{align*}
    \B[r]{\mathcal{C}}(J) &= \abs{\{ \mathcal{D} \leq \mathcal{C} \mid \dim \mathcal{D} = r, \, \mathcal{D} \leq \mathcal{C}(J)  \}}
    = \abs{\{ \mathcal{D} \leq \mathcal{C}(J) \mid \dim \mathcal{D} = r \}} \\
    &= \abs{\qbinom{\mathcal{C}(J)}{r}{q}} = \qbinom{\dim \mathcal{C}(J)}{r}{q}. \qedhere
  \end{align*}
\end{proof}

\begin{rem} \label{rem:relation_BCR_B}
  In \cite[Definition 3.3]{GeneralizedRankWeight}, the equality in Lemma~\ref{lem:B-gaussian}
  is taken as the \emph{definition} of $\B[r]{\mathcal{C}}{J}$.
  Thus our $\B[r]{\mathcal{C}}{J}$ coincides with their $B_{J}^{(r)}(\mathcal{C}, c)$ (up to the obvious notational conventions).
\end{rem}

\rev{Part~(1) of the following theorem expresses the column-support
inversion formula in the proof of \cite[Theorem~3.8]{GeneralizedRankWeight}
in terms of the rank function $\rho_{\mathcal{C}}$.
This formula is also a special case of the generalized Critical Theorem \cite{AlfaranoByrne2025}.
We include a direct proof for completeness.}

\begin{thm} \label{thm:equivalence_A_P}
  \begin{enumerate}[label=\textup{(\arabic*)}, itemsep=0ex]
    \item For every integer $0 \leq r \leq k$ and subspace $W \leq E$,
      \begin{equation*}
        \A[r]{\mathcal{C}}{W} = \sum_{J \leq W}\mu(J, W)\,\qbinom{\rho_{\mathcal{C}}(E)-\rho_{\mathcal{C}}(J^\perp)}{r}{q},
      \end{equation*}
      where we denote by $\mu\colon \mathcal{L}(E) \times \mathcal{L}(E) \to \mathbb{Z}$ the Möbius function of the subspace lattice $\mathcal L(E)$ defined as
      \begin{equation*}
        \mu(K,J)\coloneqq
        \begin{cases}
            (-1)^{\dim J-\dim K}q^{\binom{\dim J-\dim K}{2}} & \del{(K \leq J \leq E)}\rev{\text{if }K \leq J \leq E},\\
            0 & \del{(K \nleq J)}\rev{\text{otherwise}}.
        \end{cases}
    \end{equation*}
      In particular, the rank function $\rho_{C}$ \textup{(}hence $\P_{\mathcal{C}}$\del{ and
      $R_{\P_\mathcal{C}}$}\textup{)} determines the full family
      $\{ \A[r]{\mathcal{C}} \}_{r=0}^{k}$.
    \item Conversely, the distribution $\A[1]{\mathcal{C}}$ determines $\rho_{C}$ uniquely.
      More precisely, for every $J \leq E$, it holds that
      \begin{equation*}
        \rho_{\mathcal{C}}(J)=\dim \mathcal{C} - \log_{q}\!\left(1+(q-1)\sum_{W\le J^{\perp}}\A[1]{\mathcal{C}}(W)\right).
      \end{equation*}
  \end{enumerate}
  Consequently, the associated $(q,m)$-polymatroid $\P_{\mathcal{C}}$ and
  the family of higher support distributions $\{ \A[r]{\mathcal{C}} \}_{r=0}^{k}$ determine each other.
\end{thm}
\begin{proof}
  (1) Applying the M\"{o}bius inversion on $\mathcal{L}(E)$ to the straightforward equation
  \begin{equation*}
    \B[r]{\mathcal{C}}{J} = \sum_{W \leq J} \A[r]{\mathcal{C}}{W},
  \end{equation*}
  we obtain from Lemma~\ref{lem:B-gaussian} that
  \begin{equation*}
    \A[r]{\mathcal{C}}{W}
    = \sum_{J \leq W} \mu(J, W) \B[r]{\mathcal{C}}{J}
    = \sum_{J \leq W} \mu(J, W) \qbinom{\dim \mathcal{C}(J)}{r}{q}
    = \sum_{J \leq W} \mu(J, W) \qbinom{\del{\rho}\rev{\rho_{\mathcal C}}(E) - \del{\rho}\rev{\rho_{\mathcal C}}(J^{\perp})}{r}{q}.
  \end{equation*}
  Therefore, the associated $(q, m)$-polymatroid $\P_{\mathcal{C}}$ determines $\A[r]{\mathcal{C}}{W}$ for every $0 \leq r \leq k$ and every $W \leq E$.

  (2) When $r = 1$, Lemma~\ref{lem:B-gaussian} gives
  \begin{equation*}
    \B[1]{\mathcal{C}}{J} = \qbinom{\dim \mathcal{C}(J)}{1}{q} = \frac{q^{\dim \mathcal{C}(J)} - 1}{q - 1}.
  \end{equation*}
  On the other hand, we have $\displaystyle \B[1]{\mathcal{C}}{J} = \sum_{W \leq J} \A[1]{\mathcal{C}}{W}$. Hence
  \rev{\begin{equation*}
    q^{\dim\mathcal{C}(J)}
    =1+(q-1)\B[1]{\mathcal{C}}{J}
    =1+(q-1)\sum_{W\leq J}\A[1]{\mathcal{C}}{W},
  \end{equation*}
  and we obtain}
  \begin{equation*}
  \dim \mathcal{C}(J)=\log_q\left(1 + (q - 1) \sum_{W \leq J}\A[1]{\mathcal{C}}{W}\right)\rev{\in\mathbb{Z}_{\geq0}}.
  \end{equation*}
  In particular, $\dim \mathcal{C} = \dim \mathcal{C}(E)$ is recovered from $\A[1]{\mathcal{C}}$. Consequently, $\A[1]{\mathcal{C}}$ determines $\del{\rho}\rev{\rho_{\mathcal C}}$ as
  \begin{equation*}
  \del{\rho}\rev{\rho_{\mathcal C}}(J)=\del{\rho}\rev{\rho_{\mathcal C}}(E) - \dim \mathcal{C}(J^{\perp})=\dim \mathcal{C}- \log_q\left(1 + (q - 1) \sum_{W \leq J^{\perp}}\A[1]{\mathcal{C}}{W}\right).\qedhere
  \end{equation*}
\end{proof}

\rev{
\begin{rem}\label{rem:critical-theorem}
Part~(1) of Theorem~\ref{thm:equivalence_A_P} follows immediately from item~3 following \cite[Theorem~6.5]{AlfaranoByrne2025} by taking $s=1$, $C^{(1)}=\mathcal C$, $\ell_1=1$, and $d_1=r$, and writing $U=W^\perp$ and $V=J^\perp$ in the notation of that paper.
In this specialization, $\rho=\rho_{\mathcal C}$ and $k=\rho_{\mathcal C}(E)$, so $\ell_1^{-1}(k-\rho(V))=\rho_{\mathcal C}(E)-\rho_{\mathcal C}(J^\perp)$ and $\mu(U,V)=\mu(J,W)$.
Part~(2) follows by counting the zero codeword and the $q-1$ nonzero codewords in each one-dimensional subcode, as in the proof above.
\end{rem}
}
 
\section{Greene-Type Identity for Higher Rank-Weight Enumerators} \label{sect:Greene}

\rev{Throughout this section, fix a $k$-dimensional rank-metric code $\mathcal{C}\leq\Mat{n}{m}{\mathbb{F}_{q}}$.}

\rev{In this section, we give a Greene-type identity expressing the higher rank-weight enumerators in terms of the rank generating function of the associated $(q,m)$-polymatroid.
Conversely, we give an explicit formula recovering the rank generating function from the full family of higher rank-weight enumerators.
We first introduce the higher rank-weight enumerators.}

\begin{dfn}[\cite{HigherSupportWeight}] \label{dfn:higher_rank-weight_enumerator}
  For all integers $0 \leq r \leq k$, the polynomial
  \begin{equation*}
    W_{\mathcal{C}}^{(r)}(x,y)
    \coloneqq
    \sum_{\mathcal{D} \in \qbinom{\mathcal{C}}{r}{q}}
    x^{n-\wt(\mathcal{D})}\,y^{\wt(\mathcal{D})}
    \ \in\ \mathbb{Z}[x,y]
  \end{equation*}
  is called the \emph{$r$-th higher rank-weight enumerator} of $\mathcal{C}$.
\end{dfn}

\rev{Grouping the subcodes by their higher rank supports gives
\begin{equation*}
  W_{\mathcal{C}}^{(r)}(x,y)
  =\sum_{J\leq E}\A[r]{\mathcal{C}}{J}x^{n-\dim J}y^{\dim J}.
\end{equation*}
}

\begin{eg} \label{eg:higher_rank-weight_enumerator}
  Let $\mathcal{C} \coloneqq \langle M_{1}, M_{2}, M_{3} \rangle \leq \Mat{2}{3}{\mathbb{F}_{2}}$
  be the rank-metric code in Example~\ref{eg:rank_generating_function}.
  The one-dimensional subcodes are
  \begin{equation*}
    \langle M_1\rangle,\,
    \langle M_2\rangle,\,
    \langle M_3\rangle,\,
    \langle M_1+M_2\rangle,\,
    \langle M_1+M_3\rangle,\,
    \langle M_2+M_3\rangle,\,
    \langle M_1+M_2+M_3\rangle,
  \end{equation*}
  and the two-dimensional subcodes are
  \begin{align*}
    &\langle M_1,M_2\rangle,\ \langle M_1,M_3\rangle,\ \langle M_2,M_3\rangle,\ \langle  M_1,M_2+M_3\rangle,\\
    &\langle M_2,M_1+M_3\rangle,\ \langle M_3,M_1+M_2\rangle,\ \langle M_1+M_2,M_1+M_3\rangle.
  \end{align*}
  Since every nonzero codeword has column space $\mathbb F_2^2$, every nonzero subcode
  $\mathcal D\le \mathcal C$ has higher support $\Supp(\mathcal D)=\mathbb F_2^2$, and hence
  $\wt(\mathcal D)=2$. Therefore
  \begin{equation*}
    W_{\mathcal{C}}^{(0)}(x, y) = x^{2},\qquad
    W_{\mathcal{C}}^{(1)}(x, y) = 7y^{2},\qquad
    W_{\mathcal{C}}^{(2)}(x, y) = 7y^{2},\qquad
    W_{\mathcal{C}}^{(3)}(x, y) = y^{2}.
  \end{equation*}
\end{eg}

\rev{We next derive an expression for the higher rank-weight enumerators in terms of the rank generating function in Definition~\ref{dfn:rank_generating_function}.}

\begin{lem}[{\cite[Equation~(1.87)]{Stanley}}] \label{lem:q-Pochhammer}
  For all $N \in \mathbb{Z}_{\geq 0}$\rev{ and indeterminates $x,y$},
  \del{\begin{equation*}
  \sum_{s=0}^{N}
  \qbinom{N}{s}{q}
  (-1)^s q^{\binom{s}{2}}
  z^s
  =
  \prod_{i=0}^{N-1}(1-q^{i} z)\biggl(=g_N(1, z)\biggr).
\end{equation*}}
  \rev{\begin{equation*}
  \sum_{s=0}^{N}
  \qbinom{N}{s}{q}
  (-1)^s q^{\binom{s}{2}}
  x^{N-s}y^s
  =
  \prod_{i=0}^{N-1}(x-q^{i}y)
  =g_N(x,y).
\end{equation*}}
\end{lem}

\rev{This is the homogeneous form of the $q$-binomial theorem in the second edition of Volume~1 of \cite[Equation~(1.87)]{Stanley}, obtained by substituting $-y/x$ for its variable and multiplying by $x^N$.}

\del{\begin{lem} \label{lem:basis_exchange_with_Mobius}
  Let $K \in \mathcal{L}(E)$. Then,
  \begin{equation*}
    \sum_{J:\ K \leq J \leq E}
    \mu(K,J)\,x^{n-\dim J} y^{\dim J}
    =
    y^{\dim K}\,g_{n - \dim K}(x,y).
  \end{equation*}
\end{lem}
\begin{proof}
  In the subspace lattice $\mathcal{L}(E)$ of $E$, the M\"{o}bius function is calculated as
  \begin{equation*}
    \mu(K, J) = (-1)^{\dim J - \dim K} q^{\binom{\dim J - \dim K}{2}}
  \end{equation*}
  for all $K, J \in \mathcal{L}(E)$. Thus, the left-hand side of the required equation is
  \begin{align*}
    \mathrm{L.H.S.}
    &= \sum_{j = \dim K}^{n} \sum_{K \leq J \in \qbinom{E}{j}{q}} (-1)^{j - \dim K} q^{\binom{j - \dim K}{2}} x^{n - j} y^{j} \\
    &= \sum_{j = \dim K}^{n} \qbinom{n - \dim K}{j - \dim K}{q} (-1)^{j - \dim K} q^{\binom{j - \dim K}{2}} x^{n - \dim K} y^{\dim K} \left( \frac{y}{x} \right)^{j - \dim K}.
  \end{align*}

  Substituting $s \coloneqq j - \dim K$ and $N \coloneqq n - \dim K$, we have
  \begin{align*}
    \mathrm{L.H.S.} &= \sum_{s = 0}^{N} \qbinom{N}{s}{q} (-1)^{s} q^{\binom{s}{2}} x^{N} y^{\dim K} \left( \frac{y}{x} \right)^{s}
    = x^{N} y^{\dim K} \prod_{i = 0}^{N - 1} \left(1 - q^{i} \cdot \frac{y}{x} \right) \\
    &= y^{\dim K} \prod_{i = 0}^{n - \dim K - 1} (x - q^{i} y)
    = y^{\dim K} g_{n - \dim K}(x, y),
  \end{align*}

   as required.
\end{proof}}

\rev{The following expansion is given
in~\cite[Theorem~3.11]{GeneralizedRankWeight}
in terms of generalized binomial moments.
An analogous formula for $\mathbb{F}_{q^m}$-linear subcodes
is given in~\cite[Theorem~4.7]{Molina2026}.
Here we express the corresponding expansion in terms of $(q,m)$-polymatroids.}
\begin{thm} \label{thm:Greene_subsp}
  For all $r \in \{ 0, 1, \dots, k \}$, we have
  \begin{equation*}
    W_{\mathcal{C}}^{(r)}(x,y)
    =
    \sum_{J \in \mathcal{L}(E)} \qbinom{\dim \mathcal{C}(J)}{r}{q}
    y^{\dim J}\,
    g_{\,n-\dim J}(x,y),
  \end{equation*}
  or equivalently,
  \begin{equation*}
    W_{\mathcal{C}}^{(r)}(x,y)
    =
    \sum_{U \in \mathcal{L}(E)}
    \qbinom{\rho(E) -   \rho(U)}{r}{q} \, 
    y^{n-\dim U}\,
    g_{\dim U}(x,y).
  \end{equation*}
\end{thm}
\begin{proof}
\del{Substituting $\displaystyle \A[r]{\mathcal{C}}{J} = \sum_{K \leq J} \mu(K, J) \qbinom{\dim \mathcal{C}(K)}{r}{q}$  into $\displaystyle W_{\mathcal{C}}^{(r)}(x,y)$, it follows that
  \begin{align*}
     W_{\mathcal{C}}^{(r)}(x, y) &= \sum_{J \in \mathcal{L}(E)} \sum_{K \leq J} \mu(K, J) \qbinom{\dim \mathcal{C}(K)}{r}{q} x^{n - \dim J} y^{\dim J} \\
     &= \sum_{K \in \mathcal{L}(E)} \qbinom{\dim \mathcal{C}(K)}{r}{q} \sum_{J : K \leq J \leq E} \mu(K, J)  x^{n - \dim J} y^{\dim J} \\
     &= \sum_{K \in \mathcal{L}(E)} \qbinom{\dim \mathcal{C}(K)}{r}{q} y^{\dim K} g_{n - \dim K}(x, y).
  \end{align*}

  Replacing $K$ with $J$, we obtain the first claim. Putting $K = U^{\perp}$,
  \begin{equation*}
    W_{\mathcal{C}}^{(r)}(x, y)
    = \sum_{U^{\perp} \in \mathcal{L}(E)} \qbinom{\dim \mathcal{C}(U^{\perp})}{r}{q} y^{\dim   U^{\perp}} g_{n - \dim U^{\perp}}(x, y).
  \end{equation*}
  By $\rho(U) = \dim \mathcal{C} - \dim \mathcal{C}(U^{\perp})$ and $\rho(E) = \dim \mathcal{C}$, we get $\dim \mathcal{C}(U^{\perp}) = \dim \mathcal{C} - \rho(U) = \rho(E) - \rho(U)$.
  Noting that $\dim U^{\perp} = n - \dim U$, we have
  \begin{equation*}
    W_{\mathcal{C}}^{(r)}(x, y)
    =
    \sum_{U^{\perp} \in \mathcal{L}(E)} \qbinom{\rho(E) - \rho(U)}{r}{q} y^{n - \dim U} g_{\dim U}(x, y),
  \end{equation*}
  which completes the proof.}
\rev{
By Theorem~\ref{thm:equivalence_A_P}, we have
\begin{align*}
W_{\mathcal{C}}^{(r)}(x,y)
&=\sum_{K\leq E}\qbinom{\dim\mathcal{C}(K)}{r}{q}
  \sum_{K\leq J\leq E}\mu(K,J)x^{n-\dim J}y^{\dim J}\\
&=\sum_{K\leq E}\qbinom{\dim\mathcal{C}(K)}{r}{q}
  y^{\dim K}g_{n-\dim K}(x,y).
\end{align*}
The last equality follows by grouping the inner sum according to $s=\dim J-\dim K$ and applying Lemma~\ref{lem:q-Pochhammer}.
The second expression follows by setting $K=U^\perp$ and using $\dim U^\perp=n-\dim U$ and $\dim\mathcal{C}(U^\perp)=\rho_{\mathcal{C}}(E)-\rho_{\mathcal{C}}(U)$.
}
\end{proof}

\rev{We extend the higher rank-weight enumerators of codes to arbitrary $(q,m)$-polymatroids as follows.}

\rev{
\begin{dfn}\label{dfn:higher-weight-polynomial}
Let $\P=(E,\rho)$ be a $(q,m)$-polymatroid.
For every nonnegative integer $r$, we define the $r$-th higher weight polynomial of $\P$ by
\begin{equation*}
\qmpolyhigherweights{\P}{r}(x,y)\coloneqq\sum_{U\leq E}\qbinom{\rho(E)-\rho(U)}{r}{q}y^{n-\dim U}g_{\dim U}(x,y).
\end{equation*}
\end{dfn}
In particular, $\qmpolyhigherweights{\P_{\mathcal C}}{r}(x,y)=W_{\mathcal C}^{(r)}(x,y)$ for $0\leq r\leq k$.
}

\rev{To relate evaluations of the rank generating function to higher weight polynomials, we introduce the following auxiliary polynomial.}

\begin{dfn}\label{dfn:SC}
  Let $\P=(E, \rho)$ be a $(q, m)$-polymatroid.
  We define the polynomial
  \begin{equation*}
    S_{\P}(u, x, y) \coloneqq \sum_{U \leq E} u^{\rho(E) - \rho(U)} y^{n - \dim U} g_{\dim U}(x, y) \in \mathbb{Z}[u, x, y].
  \end{equation*}
\end{dfn}

\begin{lem} \label{lem:expansion_SC}
Let $\P=(E, \rho)$ be a $(q, m)$-polymatroid. Then, 
  \begin{equation*}
    S_{\P}(u,x,y)
    = y^{n-\rho(E)/m}\,R_{\P}\!
    \left( uy^{1/m},\, y^{-1/m},\, x,\, y \right).
  \end{equation*}
\end{lem}

\begin{proof}
  The definition of the rank generating function yields
  \begin{align*}
    &\del{S_{\P}(u, x, y)}\rev{y^{n-\rho(E)/m}R_{\P}\left(uy^{1/m},y^{-1/m},x,y\right)} \\
    &= y^{n - \rho(E)/m} \sum_{U \leq E} (uy^{1/m})^{\rho(E) - \rho(U)} (y^{-1/m})^{m \dim U - \rho(U)} g_{\dim U}(x, y) \\
    &= y^{n - \rho(E)/m} \sum_{U \leq E} u^{\rho(E) - \rho(U)} y^{\rho(E)/m - \dim U} g_{\dim U}(x, y) \\
    &= \sum_{U \leq E} u^{\rho(E) - \rho(U)} y^{n - \dim U} g_{\dim U}(x, y)\rev{=S_{\P}(u,x,y)}.
  \end{align*}
\rev{As shown in the third equality, the fractional powers of $y$ cancel term by term, so the resulting expression belongs to $\mathbb Z[u,x,y]$.}
\end{proof}

\begin{lem}\label{lem:q-binom-expansion}
  For all integers $1 \leq t \leq r$ and $a \in \mathbb{Z}_{\geq 0}$, we have
  \begin{equation*}
    \qbinom{a}{r}{q}
    =
    \frac{1}{g_{r}(q^{r},1)}
    \sum_{i=0}^{r}
    (-1)^{r-i}
    q^{\binom{r-i}{2}+i(t-1)}
    \qbinom{r}{i}{q}
    q^{i\max\{0,a-t+1\}}.
  \end{equation*}
\end{lem}

\begin{proof}
  If $a \geq t$, then $i(t - 1) + i\max\{0, \, a-t+1\} = ia$.
  By Lemma~\ref{lem:q-Pochhammer},
  \begin{equation*}
  \mathrm{R.H.S.}
  =
  \frac{1}{g_{r}(q^{r},1)}
  \sum_{i=0}^{r}
  (-1)^{r-i}
  q^{\binom{r-i}{2}}
  \qbinom{r}{i}{q}
  q^{ia}
  =
  \frac{g_{r}(q^a,1)}{g_{r}(q^{r}, 1)}
  =
  \qbinom{a}{r}{q}.
  \end{equation*}
  On the other hand, if $a<t$, then $\max\{0, \, a-t+1\}=0$. Therefore,
  \begin{equation*}
    \mathrm{R.H.S.}
    =
    \frac{1}{g_{r}(q^{r},1)}
    \sum_{i=0}^{r}
    (-1)^{r-i}
    q^{\binom{r-i}{2}}
    \qbinom{r}{i}{q}q^{i(t-1)}
    =
    \frac{g_{r}(q^{t-1},1)}{g_{r}(q^{r},1)} 
    =
    0
  \end{equation*}
  because $t - 1 < r$. Moreover, since $a<t\le r$, we have $\displaystyle \qbinom{a}{r}{q}=0$.
\end{proof}

\del{The next theorem is an analog of \cite[Theorem~9]{HigherSupportMatroid}.}
\rev{The following theorem gives a Greene-type identity relating higher rank-weight enumerators to the rank generating functions of the associated $(q,m)$-polymatroids and their truncations.
It is an analog of \cite[Theorem~9]{HigherSupportMatroid}.}
\begin{thm}
  For all integers $1 \leq t \leq r \leq k$,
  \begin{align*} 
    W_{\mathcal C}^{(r)}(x,y) 
    &= \frac{1}{g_{r}(q^{r},1)}  \sum_{i=0}^{r}  (-1)^{r-i}q^{\binom{r-i}{2}+i(t-1)}
    \qbinom{r}{i}{q} \,  S_{\P_{\mathcal{C}}^{(t - 1)}}\left( q^{i}, x, y \right) \\ 
    &= \frac{1}{g_{r}(q^{r},1)} \sum_{i=0}^{r}  (-1)^{r-i}q^{\binom{r-i}{2}+i(t-1)}
    \qbinom{r}{i}{q} \, y^{\,n-\frac{k-t+1}{m}} 
    R_{\P_{\mathcal{C}}^{{\,(t-1})}} \left( q^{i} y^{1/m},y^{-1/m},x,y \right)
  \end{align*}
\end{thm}
\begin{proof}
  By Theorem~\ref{thm:Greene_subsp} and Lemma~\ref{lem:q-binom-expansion}, we have
\begin{align*} 
    W_{\mathcal{C}}^{(r)}(x,y)
    &= \sum_{U \leq E}  \qbinom{\rho(E) - \rho(U)}{r}{q} \, y^{\,n-\dim U}\, g_{\dim U}(x,y) \\
    &= \sum_{U \leq E} \frac{1}{g_{r}(q^{r},1)} 
       \sum_{i=0}^{r}  (-1)^{r-i}q^{\binom{r-i}{2}+i(t-1)}
       \qbinom{r}{i}{q}\, q^{\,i\max\{0, \, k-t+1 - \rho(U) \}}\, 
       y^{\,n-\dim U}g_{\dim U}(x,y)\\
    &= \sum_{U \leq E} \frac{1}{g_{r}(q^{r},1)}
       \sum_{i=0}^{r} (-1)^{r-i}q^{\binom{r-i}{2}+i(t-1)}
       \qbinom{r}{i}{q} \\
    &\hspace{5.5cm}\times q^{\,i(k-(t-1)-\min\{\rho(U),k-(t-1)\})}\,
      y^{\,n-\dim U}g_{\dim U}(x,y)\\
    &= \frac{1}{g_{r}(q^{r}, 1)} \sum_{i=0}^{r} (-1)^{r-i}q^{\binom{r-i}{2}+i(t-1)}
      \qbinom{r}{i}{q}\, S_{\P_{\mathcal{C}}^{(t - 1)}}\left( q^{i}, x, y \right) \\
    &= \frac{1}{g_{r}(q^{r}, 1)} \sum_{i=0}^{r} (-1)^{r-i}q^{\binom{r-i}{2}+i(t-1)}
      \qbinom{r}{i}{q}\, y^{\,n-\frac{k-t+1}{m}}
      R_{\,\P_{\mathcal{C}}^{(t-1)}}\left( q^{i}y^{1/m}, y^{-1/m}, x, y \right).
\end{align*}
\rev{The last equality follows by applying Lemma~\ref{lem:expansion_SC} with $u=q^i$ to $\P_{\mathcal{C}}^{(t-1)}$, which has rank $k-t+1$.}
\end{proof}
When $t = 1$, $\P_{\mathcal{C}}^{(0)} = \P_{\mathcal{C}}$ and we obtain the next corollary. The case $r = 0$ follows directly from Theorem~\ref{thm:Greene_subsp} and Lemma~\ref{lem:expansion_SC}.
\begin{cor}
  For all integers $0 \leq r \leq k$, 
  \begin{align*}
    W_{\mathcal{C}}^{(r)}(x,y)
    &= \frac{1}{g_{r}(q^{r}, 1)} \sum_{i = 0}^{r} (-1)^{r - i} q^{\binom{r - i}{2}}
      \qbinom{r}{i}{q} y^{n - \frac{\rho(E)}{m}}
      R_{\P_{\mathcal{C}}} \!\left( q^{i}y^{1/m}, y^{-1/m}, x, y \right).
  \end{align*}
\end{cor}

\rev{
\begin{rem}
In the preceding theorem and corollary, Lemma~\ref{lem:expansion_SC} shows that the fractional powers cancel term by term in each specialization, including its prefactor, giving an element of $\mathbb Z[x,y]$ when $u=q^i$.
The full expressions, including the factor $1/g_r(q^r,1)$, also belong to $\mathbb Z[x,y]$, since they equal the higher rank-weight enumerators in Definition~\ref{dfn:higher_rank-weight_enumerator}.

\end{rem}}

\rev{
For $k\geq1$, substituting $r=1$ into the corollary gives
\begin{align*}
  (q-1)W_{\mathcal C}^{(1)}(x,y)
  &=y^{n-k/m}R_{P_{\mathcal C}}\bigl(qy^{1/m},y^{-1/m},x,y\bigr)
  -y^{n-k/m}R_{P_{\mathcal C}}\bigl(y^{1/m},y^{-1/m},x,y\bigr)\\
  &=y^{n-k/m}R_{P_{\mathcal C}}\bigl(qy^{1/m},y^{-1/m},x,y\bigr)-x^n.
\end{align*}
Counting the zero codeword and the $q-1$ nonzero codewords in each one-dimensional subcode, we obtain
\begin{equation*}
  W_{\mathcal C}(x,y)=x^n+(q-1)W_{\mathcal C}^{(1)}(x,y)=y^{n-k/m}R_{P_{\mathcal C}}\bigl(qy^{1/m},y^{-1/m},x,y\bigr).
\end{equation*}
This recovers Theorem~\ref{thm:Greene-rank}.
}

Next, we prove that the sequence $\{ W_{\mathcal{C}}^{(r)} \}_{r = 0}^{k}$
of the higher rank-weight enumerators determines the rank generating function $R_{\P_{\mathcal{C}}}$.

\begin{dfn} \label{def:BCrj}
  With a slight abuse of notation, we also define
  \begin{equation*}
    \del{\A[r]{\mathcal{C}}{w}
    \coloneqq \sum_{W \in \qbinom{E}{w}{q}} \A[r]{\mathcal{C}}{W}
    \quad\text{and}\quad}
    \B[r]{\mathcal{C}}{w}
    \coloneqq \sum_{W \in \qbinom{E}{w}{q}} \B[r]{\mathcal{C}}{W}
  \end{equation*}
  for all $w \in \{ 0, 1, \dots, n \}$.
  We refer to $\B[r]{\mathcal{C}}{w}$ as the
  \emph{$(w, r)$-th $q$-Bernstein coefficient}
  of $W_{\mathcal{C}}^{(r)}$ (cf. \cite{Lupas87}).
  For $w \in \{ 0, 1, \dots, n \}$ and $a \in \{ 0, 1, \dots, k \}$,
  we also set
  \begin{equation*}
    \N[a]{\mathcal{C}}{w}
    \coloneqq \abs{\left\{ W \in \qbinom{E}{w}{q} : \dim \mathcal{C}(W) = a \right\}}.
  \end{equation*}
\end{dfn}

\begin{lem} \label{lem:B-N}
  For all integers $w \in \{ 0, 1, \dots, n \}$ and $a, r \in \{ 0, 1, \dots, k \}$,
  \begin{equation*}
    \B[r]{\mathcal{C}}{w}
    = \sum_{a = 0}^{\dim \mathcal{C}} \N[a]{\mathcal{C}}{w} \qbinom{a}{r}{q}.
  \end{equation*}
  Conversely, 
  \begin{equation*}
    \N[a]{\mathcal{C}}{w} =
    \sum_{r=a}^{k}
    (-1)^{r - a}q^{\binom{r - a}{2}}
    \qbinom{r}{a}{q}
    \B[r]{\mathcal{C}}{w}.
  \end{equation*}
\end{lem}

\begin{proof}
  First, we have
  \begin{equation} \label{eq:explicit_generalized_binomial_moment}
    \B[r]{\mathcal{C}}{w}
    = \sum_{W \in \qbinom{E}{w}{q}} \B[r]{\mathcal{C}}{W}
    = \sum_{W \in \qbinom{E}{w}{q}} \qbinom{\dim \mathcal{C}(W)}{r}{q}.
  \end{equation}
  Grouping the terms by $a = \dim \mathcal{C}(W)$, we obtain
  \begin{equation*}
    \sum_{W \in \qbinom{E}{w}{q}} \qbinom{\dim \mathcal{C}(W)}{r}{q}
    = \sum_{a = 0}^{\dim \mathcal{C}}
      \abs{\left\{ W \in \qbinom{E}{w}{q} : \dim \mathcal{C}(W) = a \right\}}
      \cdot \qbinom{a}{r}{q} 
    = \sum_{a = 0}^{\dim \mathcal{C}} \N[a]{\mathcal{C}}{w} \qbinom{a}{r}{q},
  \end{equation*}
  which proves the first equation. \del{The second equation is obtained by applying the M\"{o}bius inversion to the first.}
  \rev{Substituting the first equation into the right-hand side of the second, the coefficient of $\N[b]{\mathcal{C}}{w}$, for $a\leq b\leq k$, is
  \begin{align*}
    \sum_{r=a}^{b}(-1)^{r-a}q^{\binom{r-a}{2}}\qbinom{r}{a}{q}\qbinom{b}{r}{q}=\qbinom{b}{a}{q}\sum_{t=0}^{b-a}(-1)^tq^{\binom{t}{2}}\qbinom{b-a}{t}{q}
    =\delta_{ab},
  \end{align*}
  where $t=r-a$ and $\delta_{ab}$ denotes the Kronecker delta. The last equality follows from Lemma~\ref{lem:q-Pochhammer} with $x=y=1$.}
\end{proof}

\begin{thm}\label{thm:inverse_Greene}
  \begin{align*}
    R_{\P_{\mathcal{C}}}(X_{1}, X_{2}, X_{3}, X_{4})
    &= \sum_{w = 0}^{n} \sum_{a = 0}^{\rho(E)} \N[a]{\mathcal{C}}{n - w} X_{1}^{a}X_{2}^{mw - \rho(E) + a} g_{w}(X_{3}, X_{4}) \\
    &= \sum_{w = 0}^{n} \sum_{a = 0}^{k} \sum_{r = a}^{k} (-1)^{r - a} q^{\binom{r - a}{2}} \qbinom{r}{a}{q} \B[r]{\mathcal{C}}{n - w} X_{1}^{a} X_{2}^{mw - k + a} g_{w}(X_{3}, X_{4}).
  \end{align*}
\end{thm}
\begin{proof}
  Since $\rho(W) = \rho(E) - \dim \mathcal{C}(W^{\perp})$ for any $W \in \mathcal{L}(E)$,
  putting $w \coloneqq \dim W$ and $a \coloneqq \dim \mathcal{C}(W^{\perp})$,
  we have
  \begin{equation*}
    \rho(E) - \rho(W) = a \quad \text{and} \quad
    m \dim W - \rho(W) = mw - \rho(E) + a.
  \end{equation*}
  Then
  \begin{equation*}
    \abs{\left\{ W \in \qbinom{E}{w}{q} : \dim \mathcal{C}(W^{\perp}) = a \right\}}
    = \abs{\left\{ U \in \qbinom{E}{n - w}{q} : \dim \mathcal{C}(U) = a \right\}}
    = \N[a]{\mathcal{C}}{n - w}
  \end{equation*}
  gives the first equation.
  \del{Applying Lemma~\ref{lem:B-N}, we have
  \begin{align*}
    R_{\P_{\mathcal{C}}}(X_{1}, X_{2}, X_{3}, X_{4})
    &= \sum_{w = 0}^{n} \sum_{a = 0}^{\rho(E)} \N[a]{\mathcal{C}}{n - w} X_{1}^{a}X_{2}^{mw - \rho(E) + a} g_{w}(X_{3}, X_{4}) \\
    &= \sum_{w = 0}^{n} \sum_{a = 0}^{k} \sum_{r = a}^{k} (-1)^{r - a} q^{\binom{r - a}{2}} \qbinom{r}{a}{q} \B[r]{\mathcal{C}}{n - w} X_{1}^{a} X_{2}^{mw - \dim \mathcal{C} + a} g_{w}(X_{3}, X_{4}) \\
    &= \sum_{a = 0}^{k} \sum_{r = a}^{k} (-1)^{r - a} q^{\binom{r - a}{2}} \qbinom{r}{a}{q} X_{1}^{a}X_{2}^{a - k} \sum_{w = 0}^{n} \B[r]{\mathcal{C}}{w} X_{2}^{m(n - w)} g_{n - w}(X_{3}, X_{4}),
  \end{align*}
  as required.}
  \rev{The second equality follows from Lemma~\ref{lem:B-N} and $\rho(E)=k$.}
\end{proof}

Once the coefficients $\B[r]{\mathcal{C}}{w}$ have been recovered from
$W_{\mathcal{C}}^{(r)}$, the polynomial $R_{\P_{\mathcal{C}}}$ is
determined by the full family of higher rank-weight enumerators $W_{\mathcal{C}}^{(r)}$.

\begin{thm}
  \begin{align*}
    &R_{\P_{\mathcal{C}}}(X_{1}, X_{2}, X_{3}, X_{4}) \\
    &=\sum_{w=0}^{n}\sum_{a=0}^{k}\left(
      \frac{1}{g_{w}(q^{w},1)}
      \sum_{r=a}^{k}\sum_{i=0}^{w}
        (-1)^{r-a+w-i}
        q^{\binom{r - a}{2}+\binom{w-i}{2}}\qbinom{r}{a}{q}\qbinom{w}{i}{q}\,
        W_{\mathcal{C}}^{(r)}(q^{i}, 1)
    \right) X_{1}^{a} X_{2}^{mw-k+a} g_{w}(X_{3},X_{4}).
  \end{align*}
  In particular, $\{ W_{\mathcal{C}}^{(r)}(x, y) \}_{r=0}^{k}$
  uniquely determines $R_{\P_{\mathcal{C}}}$.
\end{thm}

\begin{proof}
  For an arbitrary homogeneous polynomial
  \begin{equation*}
    F(x,y) = \sum_{w=0}^{n} b_{w} \, y^{w} g_{n-w}(x,y)
  \end{equation*}
  of total degree $n$, and for each $w \in \{ 0, 1, \dots, n \}$, define
  \begin{equation*}
    \Lambda_{w}(F) \coloneqq \frac{1}{g_{n-w}(q^{n-w},1)}
    \sum_{i=0}^{n-w} (-1)^{n-w-i}q^{\binom{n-w-i}{2}} \qbinom{n-w}{i}{q}\,F(q^{i},1).    
  \end{equation*}
  We claim $\Lambda_{w}(F) = b_{w}$. Indeed, set $w^{\prime} \coloneqq n - w$, and write
  \begin{equation*}
    F(u,1) = \sum_{s=0}^{n} c_{s}\,g_{s}(u,1), \qquad c_{s} = b_{n-s}.
  \end{equation*}
  Then
  \begin{equation*}
    \Lambda_{w}(F) = \frac{1}{g_{w^{\prime}}(q^{w^{\prime}},1)}
    \sum_{i=0}^{w^{\prime}} (-1)^{w^{\prime}-i}q^{\binom{w^{\prime}-i}{2}} \qbinom{w^{\prime}}{i}{q}
    \sum_{s=0}^{i} c_{s}\,g_{s}(q^{i},1).
  \end{equation*}
  Using the identities 
  $g_{s}(q^{i}, 1) = \qbinom{i}{s}{q}\,g_{s}(q^{s},1)$
  and $\qbinom{w^{\prime}}{i}{q}\qbinom{i}{s}{q} = \qbinom{w^{\prime}}{s}{q}\qbinom{w^{\prime}-s}{i-s}{q}$, 
  we obtain
  \begin{align*}
    \Lambda_{w}(F)
    &= \frac{1}{g_{w^{\prime}}(q^{w^{\prime}}, 1)} \sum_{s=0}^{w^{\prime}} c_{s} \sum_{i=s}^{w^{\prime}}
      (-1)^{w^{\prime}-i}q^{\binom{w^{\prime}-i}{2}} \qbinom{w^{\prime}}{i}{q} \qbinom{i}{s}{q} \,g_{s}(q^{s},1) \\
    &= \frac{1}{g_{w^{\prime}}(q^{w^{\prime}}, 1)} \sum_{s=0}^{w^{\prime}} c_{s}\,g_{s}(q^{s},1)\qbinom{w^{\prime}}{s}{q}
      \sum_{r=0}^{w^{\prime}-s} \qbinom{w^{\prime}-s}{r}{q} (-1)^{w^{\prime}-s-r}q^{\binom{w^{\prime}-s-r}{2}}.
  \end{align*}
  Applying Lemma~\ref{lem:q-Pochhammer} with $N \coloneqq w^{\prime} - s$ and $\del{z = 1}\rev{x=y=1}$,
  \begin{equation}
    \sum_{r=0}^{w^{\prime} - s} \qbinom{w^{\prime} - s}{r}{q} (-1)^{w^{\prime} - s - r}q^{\binom{w^{\prime} - s -r}{2}}
    =
    g_{w^{\prime} - s}(1, 1)
    =
    \begin{cases}
      1 & \del{(w^{\prime}-s=0)}\rev{\text{if }w^{\prime}-s=0},\\
      0 & \del{(w^{\prime}-s>0)}\rev{\text{otherwise}}.
    \end{cases}
  \end{equation}
  Therefore all terms vanish except the one with $s=w^{\prime}$, and hence
  \begin{equation}
    \Lambda_{w}(F)
    = \frac{1}{g_{w^{\prime}}(q^{w^{\prime}}, 1)}\,c_{w^{\prime}}\,g_{w^{\prime}}(q^{w^{\prime}},1) = c_{w^{\prime}}.
  \end{equation}
  Since $c_{s} = b_{n-s}$ and $w^{\prime} = n - w$, it follows that $\Lambda_{w}(F) = c_{w^{\prime}} = b_{w}$.
  \rev{By Theorem~\ref{thm:Greene_subsp}, grouping the subspaces by dimension gives
  \begin{equation*}
    W_{\mathcal{C}}^{(r)}(x,y)
    =\sum_{w=0}^{n}\B[r]{\mathcal{C}}{w}\,y^{w}g_{n-w}(x,y).
  \end{equation*}}
  Applying $\Lambda_{w}$ to $F = W_{\mathcal{C}}^{(r)}(x, y)$,
  we obtain $\B[r]{\mathcal{C}}(w)=\Lambda_{w}\!\left(W_{\mathcal{C}}^{(r)}\right)$.
  Substituting this into the identity in Theorem~\ref{thm:inverse_Greene}
  immediately gives the desired result.
\end{proof}

\rev{
\begin{rem}
For Hamming-metric codes, the Tutte polynomial is reconstructed as a weighted sum of higher weight enumerators with the variable substitution $(x,y)=(1,Y^{-1})$~\cite[Theorem~8]{JurriusPellikaan2009}.
Our inverse formula recovers the rank generating function by extracting expansion coefficients from the evaluations $W_{\mathcal{C}}^{(r)}(q^i,1)$ at fixed points.
\end{rem}
}
 
\section{Applications} \label{sect:Application}

\rev{In this section, we apply the higher weight polynomials of $(q,m)$-polymatroids to rank-metric codes, deriving MacWilliams-type identities for higher rank-weight enumerators and recovering the known higher rank-weight distribution of MRD codes.}

\subsection{Kl{\o}ve-type MacWilliams Identity for Rank-Metric Codes}

\rev{We derive a MacWilliams-type identity for higher rank-weight enumerators.
The argument is motivated by the relation between tuple weight enumerators and higher weight enumerators in the Hamming-metric setting; see \cite[Section~4]{Kaplan2014}.
Grouping tuples according to the subcodes they span yields the analogous relation for rank supports.
Using the $q$-transform of Gadouleau--Yan~\cite{GadouleauYan2008MacWilliamsRankMetric}, we prove a duality formula for $(q,m)$-polymatroids.
Combining this formula with the inversion relating tuple and higher weight polynomials yields the desired MacWilliams-type identity.

To extend the argument involving tuples to general $(q,m)$-polymatroids, we first introduce nonnegative integer multiples of the rank function.}

\del{
\sloppy
\noindent\textbf{Definition}~(\cite[Definitions~4 and 5]{GadouleauYan2008MacWilliamsRankMetric}).
Let $\mu$ be a nonnegative integer. Let $a(x,y;\mu)=\sum_{i=0}^{r}a_i(\mu)x^{r-i}y^i$ and $b(x,y;\mu)=\sum_{j=0}^{s}b_j(\mu)x^{s-j}y^j$ be homogeneous polynomials of degrees $r$ and $s$, respectively, where $a_i(\mu)$ and $b_j(\mu)$ are real-valued functions of $\mu$, and unspecified coefficients are understood to be zero. We define the \emph{$q$-product} by
\begin{equation*}
(a*_\mu b)(x,y;\mu)\coloneqq\sum_{u=0}^{r+s}c_u(\mu)x^{r+s-u}y^u,
\end{equation*}
where $c_u(\mu)\coloneqq\sum_{i=0}^{u}q^{is}a_i(\mu)b_{u-i}(\mu-i)$. We also define the \emph{$q$-power} by
\begin{equation*}
a^{\langle0\rangle} \coloneqq 1,\qquad a^{\langle N \rangle} \coloneqq a^{\langle N-1 \rangle}*_{\mu}a\quad (N\geq 1).
\end{equation*}
For a homogeneous polynomial $F(x,y)=\sum_{i=0}^{n}f_i x^{n-i}y^i$ of degree $n$, we define its \emph{$q$-transform} by
\begin{equation*}
\overline{F}^{(\mu)}(x,y) \coloneqq \sum_{i=0}^{n}f_i\,y^{\langle i \rangle}*_{\mu}x^{\langle n - i \rangle}.
\end{equation*}

\noindent\textbf{Lemma.}
Let $\mu$ be a nonnegative integer, and set $a(x,y;\mu)\coloneqq x+(q^\mu-1)y$, $b(x,y;\mu)\coloneqq x-y$, and $c(x,y;\mu)\coloneqq a(x,y;\mu)-b(x,y;\mu)=q^\mu y$. Then, for every nonnegative integer $s$,
\begin{equation*}
c^{\langle s \rangle}=\sum_{\ell=0}^{s}(-1)^\ell q^{\binom{\ell}{2}}\qbinom{s}{\ell}{q}b^{\langle \ell \rangle}*_{\mu}a^{\langle s - \ell \rangle}.
\end{equation*}

\begin{proof}
From \cite[Proposition~2.1]{Koornwinder1997}, it holds that
\begin{equation*}
a^{\langle s \rangle}=(b+c)^{\langle s \rangle}=\sum_{i=0}^{s}\qbinom{s}{i}{q}b^{\langle i \rangle}*_{\mu}c^{\langle s - i \rangle}. 
\end{equation*}
Applying the inversion with respect to the convolution inverse in the incidence algebra of the subspace lattice $\mathcal{L}(\mathbb F_q^s)$ gives the desired formula.
\end{proof}
}

\rev{Let $\P=(E,\rho)$ be a $(q,m)$-polymatroid.}

\begin{dfn}
For all $j \in \mathbb{Z}_{\geq 0}$, we define $\P^{[j]} \coloneqq (E,\rho^{[j]})$, where $\rho^{[j]}\colon \mathcal{L}(E)\to \mathbb Z_{\geq 0}$ is defined as $\rho^{[j]}(X) \coloneqq j \rho(X)$ for every $X\in\mathcal{L}(E)$.
\end{dfn}

\begin{prop}
For every nonnegative integer $j$, $\P^{[j]} \coloneqq (E,\rho^{[j]})$ is a $(q,mj)$-polymatroid.
\end{prop}

\begin{proof}
Since $\rho$ is a nonnegative integer-valued function and $j$ is a nonnegative integer, multiplying the inequalities (R1), (R2), and (R3) for $\rho$ by $j$ shows that $\rho^{[j]}$ satisfies (R1), (R2), and (R3).
\end{proof}

\begin{lem} \label{lem:duality-for-Pj}
For all $j\in \mathbb{Z}_{\geq 0}$,
\begin{equation*}
  \bigl( \P^{[j]} \bigr)^\ast=(\P^\ast)^{[j]}.
\end{equation*}
\rev{Here $(\P^{[j]})^\ast$ is the $mj$-dual of $\P^{[j]}$, and $\P^\ast$ is the $m$-dual of $\P$.}
\end{lem}

\begin{proof}
For every $X\in\mathcal{L}(E)$,
\begin{equation*}
  \bigl(\rho^{[j]}\bigr)^\ast(X)=j\rho(X^\perp)+mj\dim X-j\rho(E)
  =j(\rho(X^\perp)+m\dim X-\rho(E))
  =(\rho^\ast)^{[j]}(X). \qedhere
\end{equation*}
\end{proof}

\del{\noindent\textbf{Definition.}
For every nonnegative integer $r$, we define the $r$-th higher weight polynomial of a $(q,m)$-polymatroid $\P$ by
\begin{equation*}
\qmpolyhigherweights{\P}{r}(x,y)\coloneqq\sum_{U\leq E}\qbinom{\rho(E)-\rho(U)}{r}{q}y^{n-\dim U}g_{\dim U}(x,y).
\end{equation*}
}

\begin{dfn}
We define
\begin{equation*}
\qmpolyhigherweightsg{\P}(x,y)\coloneqq\sum_{U\leq E}q^{\rho(E)-\rho(U)}y^{n-\dim U}g_{\dim U}(x,y).
\end{equation*}
For every nonnegative integer $j$, we define
\begin{equation*}
\qmpolyhigherweightsg{\P}{j}(x,y) \coloneqq
\qmpolyhigherweightsg{\P^{[j]}}(x,y)=\sum_{U\leq E}q^{j(\rho(E)-\rho(U))}y^{n-\dim U}g_{\dim U}(x,y).
\end{equation*}
\end{dfn}

\rev{
\begin{lem}\label{lem:spanning-tuples}
For all nonnegative integers $a$ and $j$,
\begin{equation*}
q^{ja}=\sum_{\ell=0}^{j}\qbinom{a}{\ell}{q}g_\ell(q^j,1).
\end{equation*}
\end{lem}
\begin{proof}
The case $j=0$ is immediate. Assume $j\geq1$.
The set
\begin{equation*}
\mathcal T\coloneqq\{(\bm{v}_1,\ldots,\bm{v}_j)\mid \bm{v}_i\in\mathbb F_q^a\text{ for }1\leq i\leq j\}
\end{equation*}
has $q^{ja}$ elements.
For each $0\leq\ell\leq j$, there are $\qbinom{a}{\ell}{q}$ subspaces $D\leq\mathbb F_q^a$ of dimension $\ell$.
By \cite[Proposition~10]{Kaplan2014}, each such $D$ is spanned by $g_\ell(q^j,1)$ ordered $j$-tuples.
Grouping the elements of $\mathcal T$ according to their span proves the identity.
\end{proof}
}

\rev{The following lemma relates these polynomials to the higher weight polynomials.
It gives a $q$-analogue of the relation following \cite[Proposition~10]{Kaplan2014} for arbitrary $(q,m)$-polymatroids.}

\begin{lem}\label{lem:DualTheckeningCommutative}
For all $j\in \mathbb{Z}_{\geq 0}$,
\begin{equation*}
\qmpolyhigherweightsg{\P}{j}(x,y)=\sum_{\ell=0}^{j}g_\ell(q^j,1)\qmpolyhigherweights{\P}{\ell}(x,y)
\end{equation*}
and
\begin{equation*}
\qmpolyhigherweights{\P}{j}(x,y)=\frac{1}{g_j(q^j,1)}\sum_{\ell=0}^{j}(-1)^{j-\ell}q^{\binom{j-\ell}{2}}\qbinom{j}{\ell}{q}\qmpolyhigherweightsg{\P}{\ell}(x,y).
\end{equation*}
\end{lem}

\begin{proof}
\del{Using $q^{j(\rho(E)-\rho(U))}=\sum_{\ell=0}^{j}g_\ell(q^j,1)\qbinom{\rho(E)-\rho(U)}{\ell}{q}$, we obtain}\par
\rev{Applying Lemma~\ref{lem:spanning-tuples}, we obtain}
\del{\begin{equation*}
\qmpolyhigherweightsg{\P}{j}(x,y)\coloneqq\sum_{U\leq E}q^{j(\rho(E)-\rho(U))}y^{n-\dim U}g_{\dim U}(x,y)
\end{equation*}}
\begin{align*}
\rev{\qmpolyhigherweightsg{\P}{j}(x,y)}&=\sum_{U\leq E}\sum_{\ell=0}^{j}g_\ell(q^j,1)\qbinom{\rho(E)-\rho(U)}{\ell}{q}y^{n-\dim U}g_{\dim U}(x,y)\\
&=\sum_{\ell=0}^{j}g_\ell(q^j,1)\qmpolyhigherweights{\P}{\ell}(x,y).
\end{align*}
This proves the first formula. \del{Applying M\"obius inversion gives}\rev{Using $g_\ell(q^j,1)=\qbinom{j}{\ell}{q}g_\ell(q^\ell,1)$, we may write it as
\begin{equation*}
\qmpolyhigherweightsg{\P}{j}(x,y)
=\sum_{\ell=0}^{j}\qbinom{j}{\ell}{q}g_\ell(q^\ell,1)\qmpolyhigherweights{\P}{\ell}(x,y).
\end{equation*}
Applying M\"obius inversion on $\mathcal L(\mathbb F_q^j)$ with $\mu(U,\mathbb F_q^j)=(-1)^{j-\dim U}q^{\binom{j-\dim U}{2}}$ gives}
\begin{equation*}
g_j(q^j,1)\qmpolyhigherweights{\P}{j}(x,y)=\sum_{\ell=0}^{j}(-1)^{j-\ell}q^{\binom{j-\ell}{2}}\qbinom{j}{\ell}{q}\qmpolyhigherweightsg{\P}{\ell}(x,y).
\end{equation*}
Dividing both sides by $g_j(q^j,1)$ proves the second formula.
\end{proof}

\rev{
We now show that, for a rank-metric code $\mathcal C$ and $j\geq1$, the polynomial $\qmpolyhigherweightsg{\P_{\mathcal C}}{j}$ is the ordinary rank-weight enumerator of the code formed by horizontally concatenating $j$ codewords of $\mathcal C$.
For a $k$-dimensional rank-metric code $\mathcal C\leq\Mat{n}{m}{\mathbb F_q}$ and $j\geq1$, we identify $\mathcal C^j$ with a rank-metric code in $\Mat{n}{mj}{\mathbb F_q}$ via
\begin{equation*}
(M_1,\ldots,M_j)\longmapsto[M_1\ \cdots\ M_j].
\end{equation*}
The column support of $[M_1\ \cdots\ M_j]$ is $\sum_i\supp(M_i)=\Supp\langle M_1,\ldots,M_j\rangle$.
Since $\dim\mathcal C^j=jk$ and $\dim\mathcal C^j(U)=j\dim\mathcal C(U)$, its rank function satisfies
\begin{equation*}
\rho_{\mathcal C^j}(U)
=j\bigl(k-\dim\mathcal C(U^\perp)\bigr)
=j\rho_{\mathcal C}(U).
\end{equation*}
Thus $\P_{\mathcal C^j}=\P_{\mathcal C}^{[j]}$, and Theorem~\ref{thm:Greene-rank} gives
\begin{align*}
W_{\mathcal C^j}(x,y)
&=y^{n-\frac{jk}{mj}}R_{\P_{\mathcal C}^{[j]}}\left(qy^{1/(mj)},y^{-1/(mj)},x,y\right)
=\sum_{U\leq E}q^{j(k-\rho_{\mathcal C}(U))}y^{n-\dim U}g_{\dim U}(x,y)=\qmpolyhigherweightsg{\P_{\mathcal C}}{j}(x,y).
\end{align*}
For $j=0$, the empty tuple has support $0$ and contributes $x^n$.

\begin{rem}
For $\P=\P_{\mathcal C}$, the first identity of Lemma~\ref{lem:DualTheckeningCommutative} can also be obtained from the Critical Theorem~\cite[Theorem~4.3]{AlfaranoByrne2025}.
Indeed, for $j\geq1$ and $J\leq E$, we have
\begin{equation*}
\sum_{\ell=0}^{\min\{j,k\}}g_\ell(q^j,1)\A[\ell]{\mathcal C}{J}
=\sum_{K\leq J}\mu(K,J)q^{j(k-\rho_{\mathcal C}(K^\perp))}.
\end{equation*}
Both sides count the ordered $j$-tuples $(M_1,\ldots,M_j)\in\mathcal C^j$ whose total column support is exactly $J$.
The left-hand side groups the tuples by the subcode they span, as in the proof of Lemma~\ref{lem:spanning-tuples}.
The right-hand side is the characteristic polynomial of the corresponding minor in the Critical Theorem, evaluated at $q^j$.
Hence, we obtain
\begin{align*}
\sum_{\ell=0}^{\min\{j,k\}}g_\ell(q^j,1)W_{\mathcal C}^{(\ell)}(x,y)
&=\sum_{J\leq E}\left(\sum_{\ell=0}^{\min\{j,k\}}g_\ell(q^j,1)\A[\ell]{\mathcal C}{J}\right)
  x^{n-\dim J}y^{\dim J}\\
&=\sum_{J\leq E}\sum_{K\leq J}\mu(K,J)
  q^{j(k-\rho_{\mathcal C}(K^\perp))}x^{n-\dim J}y^{\dim J}\\
&=\sum_{K\leq E}q^{j(k-\rho_{\mathcal C}(K^\perp))}
  y^{\dim K}g_{n-\dim K}(x,y)\\
&=\qmpolyhigherweightsg{\P_{\mathcal C}}{j}(x,y).
\end{align*}
The inner sum is evaluated using Lemma~\ref{lem:q-Pochhammer}.
Since $W_{\mathcal C}^{(\ell)}=\qmpolyhigherweights{\P_{\mathcal C}}{\ell}$ for $0\leq\ell\leq k$ by Theorem~\ref{thm:Greene_subsp}, and $\qmpolyhigherweights{\P_{\mathcal C}}{\ell}=0$ for $\ell>k$ by Definition~\ref{dfn:higher-weight-polynomial}, the calculation gives the first identity for $j \geq 1$.
\end{rem}
}

\rev{We now recall the $q$-product and specify our notation for the $q$-transform.}

\rev{
\begin{dfn}[{cf.~\cite[Definitions~3 and 4]{GadouleauYan2008MacWilliamsRankMetric}}]
Let $\mu$ be an integer.
Let $a(x,y;\mu)=\sum_{i=0}^{r}a_i(\mu)x^{r-i}y^i$ and $b(x,y;\mu)=\sum_{j=0}^{s}b_j(\mu)x^{s-j}y^j$ be homogeneous polynomials of degrees $r$ and $s$, respectively, where $a_i,b_j\colon\mathbb Z\to\mathbb R$, and unspecified coefficients are understood to be zero. We define the \emph{$q$-product} by
\begin{equation*}
(a*_\mu b)(x,y;\mu)\coloneqq\sum_{u=0}^{r+s}c_u(\mu)x^{r+s-u}y^u,
\end{equation*}
where $c_u(\mu)\coloneqq\sum_{i=0}^{u}q^{is}a_i(\mu)b_{u-i}(\mu-i)$. We also define the \emph{$q$-power} by
\begin{equation*}
a^{\langle0\rangle} \coloneqq 1,\qquad a^{\langle N \rangle} \coloneqq a^{\langle N-1 \rangle}*_{\mu}a\quad (N\geq 1).
\end{equation*}

For a homogeneous polynomial $F(x,y)=\sum_{i=0}^{n}f_i x^{n-i}y^i$ with $f_i\in\mathbb R$ and homogeneous linear forms $A(x,y;\mu)$ and $B(x,y;\mu)$ whose coefficients are real-valued functions on $\mathbb Z$,
we define the \emph{$q$-transform} by
\begin{equation*}
\overline{F}^{(\mu)}[A,B]\coloneqq\sum_{i=0}^{n}f_i\,B^{\langle i\rangle}*_{\mu}A^{\langle n-i\rangle}.
\end{equation*}

\end{dfn}

\begin{rem}
Ordinary substitution does not, in general, preserve the $q$-product.
For example, substituting $x-y$ for $x$ after computing $x*_{\mu}y=xy$ gives $xy-y^2$, whereas substituting first gives $(x-y)*_{\mu}y=xy-qy^2$.
We therefore use square brackets to distinguish the operation defined above from ordinary polynomial evaluation.
For $A=x$ and $B=y$, the above definition reduces to the $q$-transform of Gadouleau--Yan \cite[Definition~4]{GadouleauYan2008MacWilliamsRankMetric}.
\end{rem}

\begin{lem}\label{lem:q-product-properties}
Let $a,b,c$ be homogeneous polynomials as above. Then,
\begin{equation*}
(a*_\mu b)*_\mu c=a*_\mu(b*_\mu c),
\qquad
1*_\mu a=a*_\mu1=a.
\end{equation*}
If $\deg b=\deg c$, then, for all $\lambda,\eta\in\mathbb R$,
\begin{align*}
a*_\mu(\lambda b+\eta c)&=\lambda(a*_\mu b)+\eta(a*_\mu c),\\
(\lambda b+\eta c)*_\mu a&=\lambda(b*_\mu a)+\eta(c*_\mu a).
\end{align*}
Moreover, for all $M,N\in\mathbb Z_{\geq0}$,
\begin{equation*}
a^{\langle M\rangle}*_\mu a^{\langle N\rangle}=a^{\langle M+N\rangle}.
\end{equation*}
\end{lem}
\begin{proof}
Let $a,b,c$ have degrees $r,s,t$, respectively. For each $u=0,\dots,r+s+t$, the coefficient of $x^{r+s+t-u}y^u$ in $(a*_\mu b)*_\mu c$ is
\begin{align*}
\sum_{v=0}^{u}q^{vt}\left(\sum_{i=0}^{v}q^{is}a_i(\mu)b_{v-i}(\mu-i)\right)c_{u-v}(\mu-v)
&=\sum_{i+j+k=u}q^{is+(i+j)t}a_i(\mu)b_j(\mu-i)c_k(\mu-i-j)\\
&=\sum_{i=0}^{u}q^{i(s+t)}a_i(\mu)\left(\sum_{j=0}^{u-i}q^{jt}b_j(\mu-i)c_{u-i-j}(\mu-i-j)\right).
\end{align*}
Here the first equality uses $v=i+j$, and the second uses $is+(i+j)t=i(s+t)+jt$. The last expression is the coefficient of $x^{r+s+t-u}y^u$ in $a*_\mu(b*_\mu c)$, proving associativity. The unit and bilinearity follow directly from the definition (cf.~\cite[Lemma~1]{GadouleauYan2008MacWilliamsRankMetric}). The last identity follows by induction using associativity.
\end{proof}

\begin{lem}\label{lem:q-binomial-inversion}
Let $\mu$ be an integer, and set $a(x,y;\mu)\coloneqq x+(q^\mu-1)y$, $b(x,y;\mu)\coloneqq x-y$, and $c(x,y;\mu)\coloneqq a(x,y;\mu)-b(x,y;\mu)=q^\mu y$. Then, for every nonnegative integer $s$,
\begin{equation*}
c^{\langle s \rangle}=\sum_{\ell=0}^{s}(-1)^\ell q^{\binom{\ell}{2}}\qbinom{s}{\ell}{q}b^{\langle \ell \rangle}*_{\mu}a^{\langle s - \ell \rangle}.
\end{equation*}
\end{lem}

\begin{proof}
By \cite[Proposition~2.1]{Koornwinder1997}, the following identity holds in the unital associative algebra generated by $X,Y$ with the relation $XY=qYX$:
\begin{equation*}
(X+Y)^s=\sum_{i=0}^{s}\qbinom{s}{i}{q}Y^iX^{s-i}.
\end{equation*}
By definition, we have
\begin{equation*}
b*_\mu c=q^\mu y(x-y),\qquad
c*_\mu b=q^{\mu+1}y(x-y)=q(b*_\mu c).
\end{equation*}
By Lemma~\ref{lem:q-product-properties}, $*_\mu$ is associative and bilinear.
Since $c*_\mu b=q(b*_\mu c)$, we may substitute $X=c$ and $Y=b$ in the above identity, taking products and powers with respect to $*_\mu$, to obtain
\begin{equation*}
a^{\langle s \rangle}=(b+c)^{\langle s \rangle}=\sum_{i=0}^{s}\qbinom{s}{i}{q}b^{\langle i \rangle}*_{\mu}c^{\langle s - i \rangle}.
\end{equation*}
\del{Applying the inversion with respect to the convolution inverse in the incidence algebra of the subspace lattice $\mathcal{L}(\mathbb F_q^s)$ gives the desired formula.}
\rev{Consider incidence convolution $\star$ on $\mathcal{L}(\mathbb F_q^s)$, using $*_\mu$ for multiplication of coefficients.
For $U\leq V$, put $d=\dim V-\dim U$ and define
\begin{equation*}
K(U,V)=b^{\langle d\rangle},\qquad
K^{-1}(U,V)=(-1)^dq^{\binom d2}b^{\langle d\rangle}.
\end{equation*}
The inverse identity follows from $b^{\langle i\rangle}*_\mu b^{\langle d-i\rangle}=b^{\langle d\rangle}$ and Lemma~\ref{lem:q-Pochhammer} with $x=y=1$.
Writing $A(U,V)=a^{\langle d\rangle}$ and $C(U,V)=c^{\langle d\rangle}$, the preceding expansion on each interval reads $A=K\star C$.
Hence $C=K^{-1}\star A$; evaluating at $(0,\mathbb F_q^s)$ gives the desired formula.}
\end{proof}
}

\begin{lem}
Let $\mu, s \in \mathbb{Z}_{\geq 0}$ with $0\leq s\leq n$. Then,
\begin{equation*}
\overline{h_s}^{(\mu)}\del{(}\rev{[}x+(q^\mu-1)y,x-y\del{)}\rev{]}=q^{\mu s}y^s g_{n-s}(x,y),
\end{equation*}
where $h_s(x,y)\coloneqq y^{n-s}g_s(x,y)$. 
\end{lem}

\begin{proof}
From Lemma~\ref{lem:q-Pochhammer}, we have
\begin{equation*}
h_s(x,y)
=y^{n-s}\sum_{a=0}^{s}(-1)^a q^{\binom a2}\qbinom{s}{a}{q}x^{s-a}y^a
=\sum_{a=0}^{s}(-1)^a q^{\binom a2}\qbinom{s}{a}{q}x^{s-a}y^{n-s+a}.
\end{equation*}
\del{Hence,}\rev{By Lemmas~\ref{lem:q-product-properties} and \ref{lem:q-binomial-inversion}, together with the definition of the $q$-product, we obtain}
\begin{align*}
\overline{h_s}^{(\mu)}\del{\bigl(}\rev{\bigl[}x+(q^\mu-1)y,x-y\del{\bigr)}\rev{\bigr]}&=\sum_{a=0}^{s}(-1)^a q^{\binom a2}\qbinom{s}{a}{q}(x-y)^{\langle n - s + a \rangle}\ast_{\mu}\bigl(x+(q^\mu-1)y\bigr)^{\langle s- a \rangle}\\
&=(x-y)^{\langle n - s \rangle}\ast_{\mu}\sum_{a=0}^{s}(-1)^a q^{\binom a2}\qbinom{s}{a}{q}(x-y)^{\langle a \rangle}\ast_{\mu}\bigl(x+(q^\mu-1)y\bigr)^{\langle s - a \rangle}\\
&=(x-y)^{\langle n - s \rangle}\ast_{\mu}\bigl(x+(q^\mu-1)y-(x-y)\bigr)^{\langle s \rangle}\\
&=(x-y)^{\langle n - s \rangle}\ast_{\mu}(q^\mu y)^{\langle s \rangle}\\
&=g_{n-s}(x,y)(q^\mu y)^s\\
&=q^{\mu s}y^s g_{n-s}(x,y). \qedhere
\end{align*}
\end{proof}

\begin{lem}\label{lem:q-transform-duality}
For all $j \in \mathbb{Z}_{\geq 0}$, 
\begin{equation*}
\qmpolyhigherweightsg{\P^\ast}{j}(x,y)
= q^{-j\rho(E)}\overline{\qmpolyhigherweightsg{\P^{[j]}}}^{(mj)}\del{(}\rev{[}x+(q^{mj}-1)y,x-y\del{)}\rev{]}.
\end{equation*}
\end{lem}

\begin{proof}
Let $Q=(E,\sigma)$ be a $(q,\mu)$-polymatroid, where $\mu$ is a nonnegative integer. Note that
\begin{equation*}
\qmpolyhigherweightsg{Q}(x,y)=\sum_{U\leq E}q^{\sigma(E)-\sigma(U)}y^{n-\dim U}g_{\dim U}(x,y)=\sum_{U\leq E}q^{\sigma(E)-\sigma(U)}h_{\dim U}(x,y).
\end{equation*}
\del{Then}\rev{By the preceding lemma,}
\begin{align*}
q^{-\sigma(E)}\overline{\qmpolyhigherweightsg{Q}}^{(\mu)}\del{\bigl(}\rev{\bigl[}x+(q^\mu-1)y,x-y\del{\bigr)}\rev{\bigr]}&=q^{-\sigma(E)}\sum_{U\leq E}q^{\sigma(E)-\sigma(U)}\overline{h_{\dim U}}^{(\mu)}\del{\bigl(}\rev{\bigl[}x+(q^\mu-1)y,x-y\del{\bigr)}\rev{\bigr]}\\
&=q^{-\sigma(E)}\sum_{U\leq E}q^{\sigma(E)-\sigma(U)}q^{\mu\dim U}y^{\dim U}g_{n-\dim U}(x,y)\\
&=\sum_{U\leq E}q^{\mu\dim U-\sigma(U)}y^{\dim U}g_{n-\dim U}(x,y).
\end{align*}
On the other hand,
\begin{align*}
\qmpolyhigherweightsg{Q^{\ast}}(x,y)&=\sum_{J\leq E}q^{\sigma^\ast(E)-\sigma^\ast(J)}y^{n-\dim J}g_{\dim J}(x,y)\\
&=\sum_{J\leq E}q^{(\mu n-\sigma(E))-(\sigma(J^\perp)+\mu\dim J-\sigma(E))}y^{n-\dim J}g_{\dim J}(x,y)\\
&=\sum_{J\leq E}q^{\mu(n-\dim J)-\sigma(J^\perp)}y^{n-\dim J}g_{\dim J}(x,y)\\
&=\sum_{U\leq E}q^{\mu\dim U-\sigma(U)}y^{\dim U}g_{n-\dim U}(x,y),
\end{align*}
where $U=J^\perp$. Taking $Q=\P^{[j]}$ and applying Lemma~\ref{lem:duality-for-Pj}, we prove the assertion.
\end{proof}

\rev{The following theorem gives a Kl{\o}ve-type MacWilliams identity~\cite[Theorem~2]{Klove} for higher weight polynomials.}

\begin{thm} \label{thm:Klove-type-MacWilliams-identity}
For all $r\in \mathbb{Z}_{\geq 0}$,
\begin{equation*}
\qmpolyhigherweights{\P^{\ast}}{r}(x,y)=\sum_{j=0}^{r}\sum_{\ell=0}^{j}\frac{(-1)^{r-j}q^{\binom{r-j}{2}-j(r-j)-\ell(j-\ell)-j\rho(E)}}{g_{r-j}(q^{r-j},1)\,g_{j-\ell}(q^{j-\ell},1)}\,\overline{\qmpolyhigherweights{\P}{\ell}}^{(mj)}\del{\bigl(}\rev{\bigl[}x+(q^{mj}-1)y,x-y\del{\bigr)}\rev{\bigr]}.
\end{equation*}
\end{thm}

\begin{proof}
Applying \del{the preceding lemmas}\rev{Lemmas~\ref{lem:DualTheckeningCommutative} and~\ref{lem:q-transform-duality}}, we obtain
\begin{align*}
\qmpolyhigherweights{\P^{\ast}}{r}(x,y)&=\frac{1}{g_r(q^r,1)}\sum_{j=0}^{r}(-1)^{r-j}q^{\binom{r-j}{2}}\qbinom{r}{j}{q}\qmpolyhigherweightsg{\P^{\ast}}{j}(x,y)\\
&=\frac{1}{g_r(q^r,1)}\sum_{j=0}^{r}(-1)^{r-j}q^{\binom{r-j}{2}-j\rho(E)}\qbinom{r}{j}{q}\overline{\qmpolyhigherweightsg{\P}{j}}^{(mj)}\del{\bigl(}\rev{\bigl[}x+(q^{mj}-1)y,x-y\del{\bigr)}\rev{\bigr]}\\
&=\sum_{j=0}^{r}\sum_{\ell=0}^{j}\frac{(-1)^{r-j}q^{\binom{r-j}{2}-j\rho(E)}}{g_r(q^r,1)}\qbinom{r}{j}{q}g_\ell(q^j,1)\,\overline{\qmpolyhigherweights{\P}{\ell}}^{(mj)}\del{\bigl(}\rev{\bigl[}x+(q^{mj}-1)y,x-y\del{\bigr)}\rev{\bigr]}\\
&=\sum_{j=0}^{r}\sum_{\ell=0}^{j}\frac{(-1)^{r-j}q^{\binom{r-j}{2}-j(r-j)-\ell(j-\ell)-j\rho(E)}}{g_{r-j}(q^{r-j},1)\,g_{j-\ell}(q^{j-\ell},1)}\,\overline{\qmpolyhigherweights{\P}{\ell}}^{(mj)}\del{\bigl(}\rev{\bigl[}x+(q^{mj}-1)y,x-y\del{\bigr)}\rev{\bigr]}. \qedhere
\end{align*}
\end{proof}

\del{Putting $\P=\P_\mathcal{C}$, we have $\qmpolyhigherweights{\P_{\mathcal{C}}}{r} = W_\mathcal{C}^{(r)}$ by Theorem~\ref{thm:Greene_subsp} and the definition of $\qmpolyhigherweights{\P_{\mathcal{C}}}{r}$. Hence we immediately obtain the following corollary.}
\rev{For $\ell>k=\dim\mathcal C$, Definition~\ref{dfn:higher-weight-polynomial} gives $\qmpolyhigherweights{\P_{\mathcal C}}{\ell}=0$, since $0\leq k-\rho_{\mathcal C}(U)\leq k$ for every $U\leq E$.
Its $q$-transform also vanishes, so the inner sum in the preceding theorem with $\P=\P_{\mathcal C}$ may be truncated at $\min\{j,k\}$.
Using $\qmpolyhigherweights{\P_{\mathcal C}}{\ell}=W_{\mathcal C}^{(\ell)}$ for $0\leq\ell\leq k$ by Theorem~\ref{thm:Greene_subsp}, we obtain the following corollary.}
\del{The following equation is a $q$-analog of \cite[Theorem~2]{Klove}.}\rev{This result is a $q$-analog of \cite[Theorem~2]{Klove}.}
\begin{cor} \label{cor:Klove-type-MacWilliams-identity}
For all $0 \leq r \leq nm - \dim \mathcal{C}$,
\begin{equation*}
W_{\mathcal{C}^\perp}^{(r)}(x,y)=\sum_{j=0}^{r}\sum_{\ell=0}^{\min\set{j, \dim \mathcal{C}}}\frac{(-1)^{r-j}q^{\binom{r-j}{2}-j(r-j)-\ell(j-\ell)-j\dim\mathcal{C}}}{g_{r-j}(q^{r-j},1)\,g_{j-\ell}(q^{j-\ell},1)}\,\overline{W_\mathcal{C}^{(\ell)}}^{(mj)}\del{\bigl(}\rev{\bigl[}x+(q^{mj}-1)y,x-y\del{\bigr)}\rev{\bigr]}.
\end{equation*}
\end{cor}

\rev{
\begin{rem}
The barred substitution in the MacWilliams formula of Gadouleau--Yan
\cite[Eq.~(17)]{GadouleauYan2008MacWilliamsRankMetric}
is to be understood as replacing $x,y$ in the defining expression
of the $q$-transform before evaluating the $q$-powers and $q$-products.
This differs from ordinary substitution into the resulting polynomial.
In our square-bracket notation, the identity reads
\begin{equation*}
W_{\mathcal C^\perp}(x,y)
=\frac{1}{|\mathcal C|}\,
\overline{W_{\mathcal C}}^{(m)}
\bigl[x+(q^m-1)y,x-y\bigr].
\end{equation*}
Under this interpretation, \cite[Eq.~(17)]{GadouleauYan2008MacWilliamsRankMetric} agrees with the explicit
$q$-product expression in
\cite[Eq.~(18)]{GadouleauYan2008MacWilliamsRankMetric}.
\end{rem}
}

\del{
\subsection{MacWilliams Identity for Generalized Binomial Moments}
We use $(q, m)$-polymatroids to give another proof of \cite[{Theorem~7.1}]{GeneralizedRankWeight}.
\begin{thm}[MacWilliams identity for generalized binomial moments]\label{thm:MacWilliams-moments}
  For all integers $0\leq j \leq n$ and $0 \leq r \leq k$,
  \begin{equation*}
    \B[r]{\mathcal{C}}{j}
    =
    \sum_{\ell=0}^{r}
    q^{\,\ell(k-m(n-j)-r+\ell)}
    \,
    \qbinom{k-m(n-j)}{r-\ell}{q}
    \B[\ell]{\mathcal{C}^{\perp}}{n-j}.
  \end{equation*}
\end{thm}

\begin{proof}
  Fix $J \in \mathcal{L}(E)$.
  From the duality $\P_{\mathcal{C}}^{\ast} = \P_{\mathcal{C}^{\perp}}$ (cf.~\cite[Proposition~11]{Shiromoto2019}),
  we obtain $\rho^{\ast}(J) = \rho(J^{\perp}) + m \dim J - \dim \mathcal{C}$.
  Substituting $\rho(J^{\perp}) = \dim \mathcal{C} - \dim \mathcal{C}(J)$,
  we have $\rho^{\ast}(J) = m \dim J - \dim \mathcal{C}(J)$.
  Combining with the calculation
  $\rho^{\ast}(J) = \dim \mathcal{C}^{\perp} - \dim \mathcal{C}^{\perp}(J^{\perp})
  = nm - k - \dim \mathcal{C}^{\perp}(J^{\perp})$
  by the definition, it follows that
  \begin{equation} \label{eq:dim_shift}
    \dim \mathcal{C}(J) = \dim \mathcal{C}^{\perp}(J^{\perp}) + k - m(n - \dim J).
  \end{equation}
  Fix $j \in \{ 0, 1, \dots, n \}$, and set $s \coloneqq k - m(n - j)$. By Equation~\eqref{eq:explicit_generalized_binomial_moment},
  \begin{equation*}
    \B[r]{\mathcal{C}}{j} = \sum_{J \in \qbinom{E}{j}{q}} \qbinom{\dim \mathcal{C}^{\perp}(J^{\perp}) + s}{r}{q}
    = \sum_{U \in \qbinom{E}{n - j}{q}} \qbinom{\dim \mathcal{C}^{\perp}(U) + s}{r}{q}.
  \end{equation*}
  Applying the $q$-Vandermonde identity
  \begin{equation*}
    \qbinom{a + b}{r}{q}
    = \sum_{i = 0}^{r} q^{i(b - r + i)} \qbinom{a}{i}{q} \qbinom{b}{r - i}{q}
  \end{equation*}
  of the extended Gaussian coefficients with $a=\dim\mathcal{C}^\perp(U)$ and $b=s$, and summing over $U$, yields
  \begin{equation*}
    \B[r]{\mathcal{C}}{j}
    = \sum_{i = 0}^{r} q^{i(s - r + i)} \qbinom{s}{r - i}{q} \sum_{U \in \qbinom{E}{n - j}{q}} \qbinom{\dim \mathcal{C}^{\perp}(U)}{i}{q}
    = \sum_{i = 0}^{r} q^{i(s - r + i)} \qbinom{s}{r - i}{q} \B[i]{\mathcal{C}^{\perp}}{n - j}
  \end{equation*}
  Substituting $s = k - m(n-j)$ concludes the proof.
\end{proof}

\begin{rem}\label{rem:n=m}
  In \cite{GeneralizedRankWeight}, when $n = m$, they define
  $\B[r]{\mathcal{C}}{j}$ as the average of the column and row versions
  in order to obtain a duality invariant.
  The proof above applies verbatim to each version (the row version
  corresponds to applying the same argument to the transposed code),
  and therefore also to their average.
\end{rem}

The inversion formula of \cite[Theorem~3.8]{GeneralizedRankWeight} gives a linear equivalence between $\{ \A[r]{\mathcal{C}}{w} \}_{w=0}^{n}$ and $\{ \B[r]{\mathcal{C}}{j} \}_{j=0}^{n}$. Therefore
Theorem~\ref{thm:MacWilliams-moments} implies the MacWilliams identity for generalized rank-weight distributions (cf. \cite[Corollary~7.4]{GeneralizedRankWeight}).
}

\subsection{MRD Codes}
\rev{Throughout this subsection, fix a $k$-dimensional rank-metric code $\mathcal{C}\leq\Mat{n}{m}{\mathbb{F}_{q}}$.}

\rev{In \cite[Lemma~4.5]{GeneralizedRankWeight}, the higher rank-weight distribution of MRD is obtained from generalized binomial moments by M\"obius inversion.
In this subsection, we derive this distribution directly from the rank function of the uniform $(q,m)$-polymatroid associated with an MRD code using Theorem~\ref{thm:Greene_subsp}.}

\del{\begin{dfn}
  Assume $m \geq n$. Let
  $d = d(C) \coloneqq \min\{\rank M \mid 0 \neq M\in \mathcal{C}\}$.
  If $k=m(n-d+1)$, $\mathcal{C}$ is called a \emph{maximum rank-distance (MRD) code}.
\end{dfn}}

\rev{Assume $m\geq n$, and let $d$ denote the minimum rank distance of $\mathcal{C}$.}

\begin{prop}[{\cite[Proposition~3.8]{q-uniform}, \cite[Corollary~6.6]{rankmetric_q-poly}}]\label{prop:MRD-uniform}
If $C$ is an MRD code, then $\P_{\mathcal{C}}$ is a uniform $(q,m)$-polymatroid with a rank function
\begin{equation*}
    \rho_{\mathcal{C}}(J) = m\min\set{\dim J, n-d+1} \qquad(J \in \mathcal{L}(E)).
\end{equation*}
\end{prop}

\begin{thm}[{\cite[Lemma~4.5]{GeneralizedRankWeight}}]\label{thm:MRD_weights}
Let $m>n$, and let $\mathcal{C}$ be an $n\times m$ MRD code. Then, for $r=0$,
\begin{equation*}
W_{\mathcal{C}}^{(0)}(x,y)=x^n,
\end{equation*}
and for integers $\del{r\geq 1}\rev{1\leq r\leq k}$,
\begin{align*}
W_{\mathcal{C}}^{(r)}(x,y)
&=
\sum_{w=d}^{n}
\qbinom{n}{w}{q}
\left(
\sum_{u=d}^{w}
(-1)^{w-u} q^{\binom{w-u}{2}}
\qbinom{w}{u}{q}
\qbinom{k-m(n-u)}{r}{q}
\right)
x^{n-w}y^w
.
\end{align*}
\end{thm}

\begin{proof}
The case $r=0$ is immediate.
Assume that $\del{r\geq 1}\rev{1\leq r\leq k}$. \del{We have}\rev{By Theorem~\ref{thm:Greene_subsp} and Proposition~\ref{prop:MRD-uniform}, we have}
\begin{align*}
W_{\mathcal{C}}^{(r)}(x,y)
&=\sum_{U\leq E}
  \qbinom{\rho(E) - \rho(U)}{r}{q} \,
  y^{\,n-\dim U}\,
  g_{\dim U}(x,y).\\
&=\sum_{j=0}^{n-d} \qbinom{n}{j}{q}\qbinom{k - mj}{r}{q}y^{n-j}g_{j}(x, y)\\
&=\sum_{j=0}^{n-d} \qbinom{n}{j}{q}\qbinom{k - mj}{r}{q}y^{n-j}\sum_{s=0}^{j}(-1)^s q^{\binom{s}{2}}\qbinom{j}{s}{q}x^{j-s}y^s\\
&=\sum_{j=0}^{n-d} \qbinom{n}{j}{q}\qbinom{k - mj}{r}{q}\sum_{w=n-j}^{n}(-1)^{w-n+j} q^{\binom{w-n+j}{2}}\qbinom{j}{\,w-n+j\,}{q}x^{n-w}y^w
\qquad (w=n-j+s)\\
&=\sum_{w=d}^{n}\left(\sum_{j=n-w}^{n-d}(-1)^{w-n+j} q^{\binom{w-n+j}{2}}\qbinom{n}{j}{q}\qbinom{j}{\,w-n+j\,}{q}\qbinom{k-mj}{r}{q}\right)x^{n-w}y^w.
\end{align*}
Using
$\qbinom{n}{j}{q}\qbinom{j}{\,w-n+j\,}{q}=\qbinom{n}{w}{q}\qbinom{w}{\,n-j\,}{q}$,
we obtain
\begin{align*}
W_{\mathcal{C}}^{(r)}(x, y)
&=
\sum_{w=d}^{n}
\qbinom{n}{w}{q}
\left(
\sum_{j=n-w}^{n-d}
(-1)^{w-n+j} q^{\binom{w-n+j}{2}}
\qbinom{w}{\,n-j\,}{q}
\qbinom{k-mj}{r}{q}
\right)
x^{n-w}y^w\\
&=
\sum_{w=d}^{n}
\qbinom{n}{w}{q}
\left(
\sum_{u=d}^{w}
(-1)^{w-u} q^{\binom{w-u}{2}}
\qbinom{w}{u}{q}
\qbinom{k-m(n-u)}{r}{q}
\right)
x^{n-w}y^w
\qquad (u=n-j).\qedhere
\end{align*}
\end{proof}

\begin{eg}
\rev{For an MRD code $\mathcal{C}$ with $m>n$ and $d=n$
\begin{align*}
W_{\mathcal{C}}^{(r)}(x,y)
&=
\sum_{w=n}^{n}
\qbinom{n}{w}{q}
\left(
\sum_{u=n}^{w}
(-1)^{w-u} q^{\binom{w-u}{2}}
\qbinom{w}{u}{q}
\qbinom{k-m(n-u)}{r}{q}
\right)
x^{n-w}y^w\\
&=\qbinom{n}{n}{q}(-1)^0q^{\binom{0}{2}}
  \qbinom{n}{n}{q}\qbinom{k-m(n-n)}{r}{q}x^{n-n}y^n=\qbinom{k}{r}{q}y^n
\qquad(1\leq r\leq m).
\end{align*}
}
\del{Let $\mathcal{C} \coloneqq \langle M_{1}, M_{2}, M_{3} \rangle \leq \Mat{2}{3}{\mathbb{F}_{2}}$
be the rank-metric code in Example~\ref{eg:rank_generating_function}.}
\del{We first show that $\mathcal{C}$ is an MRD code.
Every codeword is of the form
\begin{equation*}
M(a,b,c)=aM_1+bM_2+cM_3=
\begin{pmatrix}
a & b & c\\
c & a+c & b
\end{pmatrix}
\qquad (a,b,c\in\mathbb F_2).
\end{equation*}
If $M(a,b,c)\neq 0$, then its two rows are linearly independent, and hence
$\rank M(a,b,c)=2$. Therefore every nonzero codeword has rank $2$, so
$d(\mathcal{C})=2$. Since $\dim \mathcal{C}=3$, we obtain
\begin{equation*}
\dim \mathcal{C} = 3 = 3(2-2+1) = m(n - d + 1),
\end{equation*}
and thus $\mathcal{C}$ is an MRD code with $n=2$, $m=3$, $k=3$, and $d=2$.}
\del{Then, Theorem~\ref{thm:MRD_weights} gives
\begin{align*}
W_{\mathcal{C}}^{(r)}(x,y)
&=
\sum_{w=2}^{2}
\qbinom{2}{w}{2}
\left(
\sum_{u=2}^{w}
(-1)^{w-u} 2^{\binom{w-u}{2}}
\qbinom{w}{u}{2}
\qbinom{3-3(2-u)}{r}{2}
\right)
x^{2-w}y^w\\
&=\qbinom{3}{r}{2}y^2
\end{align*}
for each $r\in \{ 1, 2, 3 \}$. Thus}
\rev{The code in Example~\ref{eg:rank_generating_function} satisfies these conditions with $n=d=2$ and $m=3$, so we immediately obtain}
\begin{equation*}
W_{\mathcal{C}}^{(1)}(x,y)=7y^2,\qquad
W_{\mathcal{C}}^{(2)}(x,y)=7y^2,\qquad
W_{\mathcal{C}}^{(3)}(x,y)=y^2,
\end{equation*}
which agree with the direct computation in Example~\ref{eg:higher_rank-weight_enumerator}.
\end{eg}

\section*{Acknowledgments}

\rev{The authors would like to thank the anonymous reviewers for their very helpful comments and suggestions that improved the paper.}
\del{The authors would like to thank the participants of related seminars and discussions
for their helpful comments and suggestions.}

\section*{Funding}
This work was supported by
the Japan Society for the Promotion of Science (JSPS) KAKENHI
Grants JP25K17298 and JP25K07103, 
and the Japan Science and Technology Agency (JST) BOOST Grants JPMJBS2408.

\section*{Statements and Declarations}
\subsection*{Competing interests}
The authors have no relevant financial or non-financial interests to disclose. 

\bibliography{reference}

@article{Barg2002,
    author = {{\rev{A}} {\rev{Barg}}},
    title = {\rev{On some polynomials related to weight enumerators of linear codes}},
    journal = {\rev{SIAM J. Discrete Math.}},
    volume = {\rev{15}},
    number = {\rev{2}},
    pages = {\rev{155--164}},
    year = {\rev{2002}},
    doi = {10.1137/S0895480199364148},
    url = {https://doi.org/10.1137/S0895480199364148},
}

@book{Stanley,
  author = {Stanley, Richard P.},
  title = {Enumerative Combinatorics},
  volume = {\del{2}\rev{1}},
  edition = {\rev{second}},
  publisher = {Cambridge University Press},
  year = {\del{1999}\rev{2012}},
  isbn = {9781107015425},
  doi = {10.1017/CBO9781139058520},
  url = {https://doi.org/10.1017/CBO9781139058520}
}

@article {GeneralizedRankWeight,
    AUTHOR = {Byrne, Eimear and Cotardo, Giuseppe and Ravagnani, Alberto},
     TITLE = {Rank-metric codes, generalized binomial moments and their zeta
              functions},
   JOURNAL = {Linear Algebra Appl.},
  FJOURNAL = {Linear Algebra and its Applications},
    VOLUME = {604},
      YEAR = {2020},
     PAGES = {92--128},
      ISSN = {0024-3795,1873-1856},
   MRCLASS = {11T71 (05A17 11M41)},
  MRNUMBER = {4116823},
MRREVIEWER = {Steven\ T.\ Dougherty},
       DOI = {10.1016/j.laa.2020.06.002},
       URL = {https://doi.org/10.1016/j.laa.2020.06.002},
}

@misc{GLJ25MonomialSubstitution,
      title={Polynomial Invariants of $q$-Matroids and Rank-Metric Codes}, 
      author={Heide Gluesing-Luerssen and Benjamin Jany},
      year={2025},
      eprint={2509.21618},
      archivePrefix={arXiv},
      primaryClass={math.CO},
      note={arXiv:2509.21618},
      url={https://arxiv.org/abs/2509.21618}, 
}

@misc{ABN2026,
      title={\rev{The geometry of rank-metric codes}},
      author={{\rev{G.~N}} {\rev{Alfarano}} and {\rev{M}} {\rev{Borello}} and {\rev{A}} {\rev{Neri}}},
      year={\rev{2026}},
      eprint={2605.19691},
      archivePrefix={arXiv},
      primaryClass={math.CO},
      note={\rev{arXiv:2605.19691}},
      url={https://arxiv.org/abs/2605.19691v1},
}

@misc{Molina2026,
      title={\rev{{MacWilliams} identities for the generalized rank weights}},
      author={{\rev{J}} {\rev{Molina}}},
      year={\rev{2026}},
      eprint={2602.10929},
      archivePrefix={arXiv},
      primaryClass={cs.IT},
      note={\rev{arXiv:2602.10929}},
      url={https://arxiv.org/abs/2602.10929v2},
}

@article{AlfaranoByrne2025,
      title={\rev{Recursive properties of the characteristic polynomial of weighted lattices}},
      author={{\rev{G.~N}} {\rev{Alfarano}} and {\rev{E}} {\rev{Byrne}}},
      journal={\rev{Advances in Applied Mathematics}},
      volume={\rev{176}},
      pages={\rev{103046}},
      year={\del{2025}\rev{2026}},
      doi={10.1016/j.aam.2026.103046},
      eprint={2305.07567},
      archivePrefix={arXiv},
      primaryClass={math.CO},
      note={\rev{arXiv:2305.07567v3 (theorem numbering follows this version)}},
      url={https://doi.org/10.1016/j.aam.2026.103046},
}

@article {Shiromoto2019,
    AUTHOR = {Shiromoto, Keisuke},
     TITLE = {Codes with the rank metric and matroids},
   JOURNAL = {Des. Codes Cryptogr.},
  FJOURNAL = {Designs, Codes and Cryptography. An International Journal},
    VOLUME = {87},
      YEAR = {2019},
    NUMBER = {8},
     PAGES = {1765--1776},
      ISSN = {0925-1022,1573-7586},
   MRCLASS = {94B60 (05B35)},
  MRNUMBER = {3974800},
       DOI = {10.1007/s10623-018-0576-0},
       URL = {https://doi.org/10.1007/s10623-018-0576-0},
}

@article {HigherSupportMatroid,
    AUTHOR = {Britz, Thomas},
     TITLE = {Higher support matroids},
   JOURNAL = {Discrete Math.},
  FJOURNAL = {Discrete Mathematics},
    VOLUME = {307},
      YEAR = {2007},
    NUMBER = {17-18},
     PAGES = {2300--2308},
      ISSN = {0012-365X,1872-681X},
   MRCLASS = {05B35 (94B05)},
  MRNUMBER = {2340630},
MRREVIEWER = {Keisuke\ Shiromoto},
       DOI = {10.1016/j.disc.2006.12.001},
       URL = {https://doi.org/10.1016/j.disc.2006.12.001},
}

@article {Greene,
    AUTHOR = {Greene, Curtis},
     TITLE = {Weight enumeration and the geometry of linear codes},
   JOURNAL = {Studies in Appl. Math.},
  FJOURNAL = {Studies in Applied Mathematics},
    VOLUME = {55},
      YEAR = {1976},
    NUMBER = {2},
     PAGES = {119--128},
      ISSN = {0022-2526,1467-9590},
   MRCLASS = {05B35 (94A10)},
  MRNUMBER = {447020},
MRREVIEWER = {Thomas\ Brylawski},
       DOI = {10.1002/sapm1976552119},
       URL = {https://doi.org/10.1002/sapm1976552119},
}

@article {HigherSupportWeight,
    AUTHOR = {Britz, Thomas},
     TITLE = {Code enumerators and {T}utte polynomials},
   JOURNAL = {IEEE Trans. Inform. Theory},
  FJOURNAL = {Institute of Electrical and Electronics Engineers.
              Transactions on Information Theory},
    VOLUME = {56},
      YEAR = {2010},
    NUMBER = {9},
     PAGES = {4350--4358},
      ISSN = {0018-9448,1557-9654},
   MRCLASS = {94A05 (05C31 94B05)},
  MRNUMBER = {2807332},
       DOI = {10.1109/TIT.2010.2054654},
       URL = {https://doi.org/10.1109/TIT.2010.2054654},
}

@inproceedings{JurriusPellikaan2009,
    author = {{\rev{R}} {\rev{Jurrius}} and {\rev{R}} {\rev{Pellikaan}}},
    title = {\rev{Extended and Generalized Weight Enumerators}},
    booktitle = {\rev{Proceedings of the International Workshop on Coding and Cryptography ({WCC} 2009)}},
    year = {\rev{2009}},
    pages = {\rev{76--91}},
    url = {https://ruudp.win.tue.nl/paper/53.pdf},
}

@article {WeightedSubspaceDesign,
    AUTHOR = {Byrne, Eimear and Ceria, Michela and Ionica, Sorina and
              Jurrius, Relinde},
     TITLE = {Weighted subspace designs from {$q$}-polymatroids},
   JOURNAL = {J. Combin. Theory Ser. A},
  FJOURNAL = {Journal of Combinatorial Theory. Series A},
    VOLUME = {201},
      YEAR = {2024},
     PAGES = {Paper No. 105799, 40},
      ISSN = {0097-3165,1096-0899},
   MRCLASS = {05B35 (05B05)},
  MRNUMBER = {4631380},
MRREVIEWER = {Keisuke\ Shiromoto},
       DOI = {10.1016/j.jcta.2023.105799},
       URL = {https://doi.org/10.1016/j.jcta.2023.105799},
}

@inproceedings{Lupas87,
  title={A $q$-analogue of the {Bernstein} operator},
  author={Lupas, A},
  booktitle={Seminar on numerical and statistical calculus, University of Cluj-Napoca},
  volume={9},
  number={85-92},
  year={1987}
}

@article {q-uniform,
    AUTHOR = {Gluesing-Luerssen, Heide and Jany, Benjamin},
     TITLE = {{$q$}-polymatroids and their relation to rank-metric codes},
   JOURNAL = {J. Algebraic Combin.},
  FJOURNAL = {Journal of Algebraic Combinatorics. An International Journal},
    VOLUME = {56},
      YEAR = {2022},
    NUMBER = {3},
     PAGES = {725--753},
      ISSN = {0925-9899,1572-9192},
   MRCLASS = {05B35 (94B05)},
  MRNUMBER = {4491058},
MRREVIEWER = {Keisuke\ Shiromoto},
       DOI = {10.1007/s10801-022-01129-y},
       URL = {https://doi.org/10.1007/s10801-022-01129-y},
}

@article {rankmetric_q-poly,
    AUTHOR = {Gorla, Elisa and Jurrius, Relinde and L\'opez, Hiram H. and
              Ravagnani, Alberto},
     TITLE = {Rank-metric codes and {$q$}-polymatroids},
   JOURNAL = {J. Algebraic Combin.},
  FJOURNAL = {Journal of Algebraic Combinatorics. An International Journal},
    VOLUME = {52},
      YEAR = {2020},
    NUMBER = {1},
     PAGES = {1--19},
      ISSN = {0925-9899,1572-9192},
   MRCLASS = {94B60 (05A30 05B35)},
  MRNUMBER = {4128042},
MRREVIEWER = {N.\ L.\ Manev},
       DOI = {10.1007/s10801-019-00889-4},
       URL = {https://doi.org/10.1007/s10801-019-00889-4},
}

@article{RAVAGNANI2016,
TITLE = {Generalized weights: An anticode approach},
JOURNAL = {Journal of Pure and Applied Algebra},
VOLUME = {220},
NUMBER = {5},
PAGES = {1946-1962},
YEAR = {2016},
ISSN = {0022-4049},
DOI = {https://doi.org/10.1016/j.jpaa.2015.10.009},
URL = {https://www.sciencedirect.com/science/article/pii/S0022404915002893},
AUTHOR = {Alberto Ravagnani},
}

@article{MartinezPenasMatsumoto2018,
    author = {{\rev{U}} {\rev{Mart{\'i}nez-Pe{\~n}as}} and {\rev{R}} {\rev{Matsumoto}}},
    title = {\rev{Relative generalized matrix weights of matrix codes for universal security on wire-tap networks}},
    journal = {\rev{IEEE Trans. Inform. Theory}},
    volume = {\rev{64}},
    number = {\rev{4}},
    pages = {\rev{2529--2549}},
    year = {\rev{2018}},
    doi = {10.1109/TIT.2017.2766292},
    url = {https://doi.org/10.1109/TIT.2017.2766292},
}

@article{BritzMammolitiShiromoto2020,
    author = {{\rev{T}} {\rev{Britz}} and {\rev{A}} {\rev{Mammoliti}} and {\rev{K}} {\rev{Shiromoto}}},
    title = {\rev{{Wei}-type duality theorems for rank metric codes}},
    journal = {\rev{Des. Codes Cryptogr.}},
    volume = {\rev{88}},
    number = {\rev{8}},
    pages = {\rev{1503--1519}},
    year = {\rev{2020}},
    doi = {10.1007/s10623-019-00688-9},
    url = {https://doi.org/10.1007/s10623-019-00688-9},
}

@article {demipoly2020,
    AUTHOR = {Ghorpade, Sudhir R. and Johnsen, Trygve},
     TITLE = {A polymatroid approach to generalized weights of rank metric
              codes},
   JOURNAL = {Des. Codes Cryptogr.},
  FJOURNAL = {Designs, Codes and Cryptography. An International Journal},
    VOLUME = {88},
      YEAR = {2020},
    NUMBER = {12},
     PAGES = {2531--2546},
      ISSN = {0925-1022,1573-7586},
   MRCLASS = {05B35 (15A03 94B60)},
  MRNUMBER = {4171316},
MRREVIEWER = {Maruti\ M.\ Shikare},
       DOI = {10.1007/s10623-020-00798-9},
       URL = {https://doi.org/10.1007/s10623-020-00798-9},
}

@incollection{Koornwinder1997,
  AUTHOR    = {Tom H. Koornwinder},
  TITLE     = {Special functions and q-commuting variables},
  BOOKTITLE = {Special Functions, q-Series and Related Topics},
  EDITOR    = {M. E. H. Ismail and D. R. Masson and M. Rahman},
  SERIES    = {Fields Institute Communications},
  VOLUME    = {14},
  PAGES     = {131--166},
  PUBLISHER = {American Mathematical Society},
  YEAR      = {1997}
}

@article{GadouleauYan2008MacWilliamsRankMetric,
  author  = {Gadouleau, Maximilien and Yan, Zhiyuan},
  title   = {\del{{MacWilliams} Identity for the Rank Metric}\rev{{MacWilliams} Identity for Codes with the Rank Metric}},
  journal = {\del{In Proceedings of the 2007 IEEE International Symposium on Information Theory}\rev{EURASIP Journal on Wireless Communications and Networking}},
  volume  = {\rev{2008}},
  year    = {\del{2007}\rev{2008}},
  pages   = {\del{2486--2490}\rev{Article ID 754021, 13 pages}},
  doi     = {10.1155/2008/754021},
  url     = {https://doi.org/10.1155/2008/754021}
}

@article{klove,
title = {Support weight distribution of linear codes},
journal = {Discrete Mathematics},
volume = {106-107},
pages = {311-316},
year = {1992},
issn = {0012-365X},
doi = {https://doi.org/10.1016/0012-365X(92)90559-X},
url = {https://www.sciencedirect.com/science/article/pii/0012365X9290559X},
author = {Torleiv Kløve}
}

@article{Delsarte1978,
  author = {{\rev{P}} {\rev{Delsarte}}},
  title = {\rev{Bilinear forms over a finite field, with applications to coding theory}},
  journal = {\rev{J. Combin. Theory Ser. A}},
  volume = {\rev{25}},
  number = {\rev{3}},
  pages = {\rev{226--241}},
  year = {\rev{1978}},
  doi = {10.1016/0097-3165(78)90015-8},
  url = {https://doi.org/10.1016/0097-3165(78)90015-8},
}

@article{Gabidulin1985,
  author = {{\rev{E.~M}} {\rev{Gabidulin}}},
  title = {\rev{Theory of codes with maximum rank distance}},
  journal = {\rev{Problems Inform. Transmission}},
  volume = {\rev{21}},
  number = {\rev{1}},
  pages = {\rev{1--12}},
  year = {\rev{1985}},
  url = {https://www.mathnet.ru/eng/ppi967},
}

@article{Kaplan2014,
  author = {{\rev{N}} {\rev{Kaplan}}},
  title = {\rev{{MacWilliams} identities for {$m$}-tuple weight enumerators}},
  journal = {\rev{SIAM J. Discrete Math.}},
  volume = {\rev{28}},
  number = {\rev{1}},
  pages = {\rev{428--444}},
  year = {\rev{2014}},
  doi = {10.1137/120876356},
  url = {https://doi.org/10.1137/120876356},
}
\bibliographystyle{abbrv}

\end{document}